\documentclass[a4paper]{amsart}
\usepackage{amsmath}
\usepackage{hyperref}

\usepackage{color}

\newcommand{\R}{\mathbb{R}}

\newcommand{\C}{\mathbb{C}}
\newcommand{\I}{\mathbb{I}}
\newcommand{\settc}[2]{\bigl\{\,#1 \bigm\vert #2\,\bigr\}}

\newcommand{\norm}[1]{\lVert #1 \rVert}

\newcommand{\free}{D_{0}}
\newcommand{\freefw}{D_{_{\text{FW}}}}
\newcommand{\fw}{U_{_{\text{FW}}}}
\newcommand{\Rfw}{U^{R}_{_{\text{FW}}}}
\newcommand{\Vfw}{V^{^{2 \times 2}}_{_\text{FW}}}

\DeclareMathOperator{\supess}{sup\,ess}
\DeclareMathOperator{\esp}{e}
\DeclareMathOperator{\RE}{Re}

\newcommand{\abs}[1]{\lvert #1 \rvert}
\newcommand{\labs}[1]{\left| #1 \right|}
\newtheorem{thm}[equation]{Theorem}

\newtheorem{prop}[equation]{Proposition}
\newtheorem{lem}[equation]{Lemma}

\theoremstyle{remark}
\newtheorem{rem}[equation]{Remark}
\newtheorem*{notation}{Notation}

\numberwithin{equation}{section}

\title[Brown-Ravenhall operator]{A variational approach to the
Brown-Ravenhall operator for the relativistic one-electron atoms}

\author{Vittorio Coti Zelati}
\email[Coti Zelati]{zelati@unina.it}
\address[Coti Zelati]{Dipartimento di Matematica Pura e Applicata 
``R.~Caccioppoli''\\
Universit\`a di Napoli ``Federico II''\\
via Cintia, M.S.~Angelo\\
80126 Napoli (NA), Italy}

\author{Margherita Nolasco}
\email[Nolasco]{nolasco@univaq.it}
\address[Nolasco]{Dipartimento di Ingegneria e Scienze dell' informazione e Matematica ,
Universit\`a dell'Aquila
via Vetoio, Loc. Coppito	
67010 L'Aquila AQ
Italia}

\thanks{Research partially supported by MIUR grant PRIN 201274FYK7, "Variational and perturbative aspects of nonlinear differential problems". One of the Authors (V. Coti Zelati) was also partially supported by Programme STAR, UniNA and Compagnia di San Paolo}

\begin{document}

\begin{abstract} 
We use the  Foldy--Wouthuysen (unitary) transformation to  give an alternative characterization of the eigenvalues and eigenfunctions for the Brown-Ravenhall operator (the projected Dirac operator)  in the case of a one-electron atom. 
In particular we 
 transform  the eigenvalues problem 
into an elliptic problem in the 4-dim half space  $\R^4_{+}$ with  Neumann  boundary condition. 
\end{abstract} 

\maketitle

\section{Introduction and main results}

The Dirac operator is a first order operator acting on 4-spinors $\Psi
\colon \mathbb{R}^{3} \to \mathbb{C}^{4}$, given by
\begin{equation*}
    D_{0}= - i c \hbar \underline{\mathbf{\alpha}} \cdot \nabla + m
    c^2 \mathbf{\beta}
\end{equation*}
where $c$ denotes the speed of light, $m >0$ the mass, $\hbar$ the
Planck's constant  (from now on we choose a system of physical units such that $\hbar =1$), $\mathbf{\alpha}_k $, $k=1,2,3$ and
$\mathbf{\beta}$ are the Pauli-Dirac $4 \times 4$-matrices,
\begin{equation*}
    \mathbf{\beta } = 
    \begin{pmatrix} 
	{\I}_{_2} & 0 \\ 
	0 & - \I_{_2} 
    \end{pmatrix} 
    \qquad \mathbf{\alpha}_k =
    \begin{pmatrix} 
	0 & \mathbf{ \sigma}_k \\ 
	\mathbf{ \sigma}_k & 0
    \end{pmatrix}  \qquad k=1,2,3
\end{equation*}
and $\mathbf{\sigma}_k$ are the Pauli $2\times 2$-matrices given by
\begin{equation*}
    \sigma_{1} =
    \begin{pmatrix}
	0 & 1 \\
        1 & 0
    \end{pmatrix},
    \quad \sigma_{2} =
    \begin{pmatrix}
	0 & -i \\
        i & 0
    \end{pmatrix}
    \quad \sigma_{3} =
    \begin{pmatrix}
	1 & 0 \\
        0 & -1
    \end{pmatrix}.
\end{equation*}

Denoting the Fourier transform (of a function in $u \in
\mathcal{S}(\mathbb{R}^{3})$) by
\begin{equation*}
    \mathcal{F}(u)(p) = \frac{1}{(2\pi)^{3/2}} \int_{\mathbb{R}^{3}}
    e^{-ip \cdot x} u(x) \, dx,
\end{equation*}
the free Dirac operator becomes in (momentum) Fourier space the
multiplication operator $\hat{D}(p) = \mathcal{F} D_{0}
\mathcal{F}^{-1} (p)$ given, for each $p \in \R^3$, by an Hermitian $4
\times 4$-matrix which has the eigenvalues
\begin{equation*}
    \lambda_1(p) = \lambda_2(p) = - \lambda_3(p)= - \lambda_4(p) =
    \sqrt{c^2 |p|^2 + m^2 c^4 } \equiv \lambda(p).
\end{equation*}

The unitary transformation $U(p)$ which diagonalize $\hat{D}(p)$ is
given explicitly by
\begin{align*}
    & U(p) = a_{+} (p) \I_{4} + a_{-} (p) \mathbf{\beta}
    \frac{\underline{\mathbf{\alpha}} \cdot p}{|p|} \\
    &U^{-1}(p) = a_{+} (p) \I_{4} - a_{-} (p) \mathbf{\beta}
    \frac{\underline{\mathbf{\alpha}} \cdot p}{|p|}
\end{align*}
with $a_{\pm}(p) = \sqrt{\frac{1}{2}( 1 \pm mc^2 / \lambda(p))}$
\begin{equation*}
    U(p)\hat{D}(p) U^{-1}(p) = \mathbf{\beta} \lambda(p) =
    \begin{pmatrix} 
	\I_{_2} & 0 \\ 
	0 & - \I_{_2} 
    \end{pmatrix}
    \sqrt{c^2 |p|^2 + m^2 c^4 }
\end{equation*}

We recall here the main properties of the free Dirac operator $\free$.
\begin{prop}[\protect{see \cite[chaper 1]{Thaller_1992}}]
    The free Dirac operator $\free$ is essentially self-adjoint  and self-adjoint on $\mathcal{D}(\free) =
    H^{1}(\mathbb{R}^{3};\mathbb{C}^{4})$.
    
    Its spectrum is purely absolutely continuous and given by
    \begin{equation*}
	\sigma(\free) = (-\infty,-mc^{2}] \cup [mc^{2}, +\infty).
    \end{equation*}
    
    There are two orthogonal projectors on $L^{2}(\mathbb{R}^{3},
    \mathbb{C}^{4})$, both infinite rank,
    \begin{equation*}
	\Lambda_{\pm} = \mathcal{F}^{-1}U(p)^{-1} \left( \frac{\I_{4}
	\pm \mathbf{\beta}}{2} \right) U(p)\mathcal{F}
    \end{equation*}
    and such that
    \begin{equation*}
	\free \Lambda_{\pm} = \Lambda_{\pm} \free = \pm \sqrt{- c^2 \Delta + m^{2}c^{4}} \Lambda_{\pm} = \pm
	\Lambda_{\pm} \sqrt{- c^2 \Delta + m^{2}c^{4}}\I_{4}.
    \end{equation*}
\end{prop}

The operator $\sqrt{-c^{2}\Delta + m^2c^{4}}$ can be defined
for all $f \in H^{1}(\mathbb{R}^{N})$ as the inverse Fourier transform
of the $L^{2}$ function $\sqrt{c^{2}|p|^{2} + m^2c^{4}} \, \hat{f}
(p)$ (where $\hat{f} = \mathcal{F}(f)$).

To such an operator we can also associate the following quadratic form
\begin{equation*}
    \mathcal{Q}(f,g) = \int_{\mathbb{R}^{N}} \sqrt{ c^{2}|p|^{2} +
    m^2c^{4}} \, \, \hat{f} (p) \hat{g}(p) \, dp
\end{equation*}
which can be extended to all functions $f, g \in
H^{1/2}(\mathbb{R}^{N})$ where
\begin{equation*}
    H^{1/2}(\mathbb{R}^{N})= \settc{f \in L^{2}(\mathbb{R}^{N})}
    {\int_{\mathbb{R}^{N}} (1 + | p|) |\hat{f} (p)|^{2} \, dp <
    +\infty}.
\end{equation*}
see for example \cite{LiebLoss} for more details.

We now introduce the \textbf{Foldy-Wouthuysen} transformation, given
by the unitary transformation $\fw = \mathcal{F}^{-1} U(p)
\mathcal{F}$ which transforms the free Dirac operator into the $2
\times 2$-block form $D_{_\text{FW}} = \fw D_{0}\fw^{-1}$
\begin{equation*}
    D_{_\text{FW}}= 
    \begin{pmatrix} 
	\sqrt{- c^2 \Delta + m^2 c^4 } \, \I_{_2} & 0 \\
	0& - \sqrt{- c^2 \Delta + m^2 c^4 } \, \I_{_2}
    \end{pmatrix} 
\end{equation*}

Under this transformation the projectors $\Lambda_{\pm}$ become simply
\begin{equation*}
    \fw \Lambda_{\pm} \fw^{-1} = \frac{\I_{4} \pm \mathbf{\beta}}{2}
\end{equation*}
therefore the positive and negative energy subspaces for $\freefw$,
are simply given by
\begin{align*}
    &\mathcal{H}_{+} = \settc{\Psi = \left(
    \begin{smallmatrix}
	\psi\\
	0
    \end{smallmatrix}
    \right) \in L^{2}(\mathbb{R}^{3},\mathbb{C}^{4}) }{ \psi \in
    L^{2}(\mathbb{R}^{3},\mathbb{C}^{2})}\\
    &\mathcal{H}_{-} = \settc{\Psi = \left(
    \begin{smallmatrix}
	0 \\
	\psi 
    \end{smallmatrix}
    \right) \in L^{2}(\mathbb{R}^{3},\mathbb{C}^{4}) }{ \psi \in
    L^{2}(\mathbb{R}^{3},\mathbb{C}^{2})}.
\end{align*}

We are interested in perturbed Dirac operators $D_{0} + V$, $V$ being
a scalar potential satisfying
\begin{enumerate}
    \item[(\textbf{h1})] $V \in {L_{w}^{3}(\R^3)} + L^{\infty}(\R^3)$;
    
    \item[(\textbf{h2})] there exists $a \in (0,1)$ such that
    \begin{equation*}
	\left| \left(\Lambda_{+} \phi, V \Lambda_{+} \phi
	\right)_{_{L^2}} \right| \leq a \left(\Lambda_{+} \phi, D_{0}
	\, \Lambda_{+} \phi \right)_{_{L^2}}
    \end{equation*}
    for all $\phi \in H^{1/2}(\R^3; \mathbb{C}^{4})$;
    
    \item[(\textbf{h3})] $V \in L^{\infty}(\R^3 \setminus
    \overline{B}_{R_0})$ for some $R_0>0$ and
    \begin{align*}
	(i) & \lim_{R \to + \infty} \| V \|_{L^{\infty}(|x| >R)} =
	0 ; \\
	(ii) & \lim_{R \to + \infty} \supess_{|x| >R} V (x) |x|^2 = -
	\infty .
    \end{align*}
\end{enumerate}

\begin{rem}
    We recall that $L^q_{w}(\mathbb{R}^N)$, the weak $L^{q}$ space, is
    the space of all measurable functions $f$ such that
    \begin{equation*}
	\sup_{\alpha >0 } \alpha | \settc{ x }{|f(x)| > \alpha}|^{1/q}
	< + \infty,
    \end{equation*}
    where $|E|$ denotes the Lebesgue measure of a set $E \subset
    \mathbb{R}^N $.  Note that $ f(x)= |x|^{-1}$ does not belong to
    any $L^q$-space but it belongs to $ L^3_{w}(\mathbb{R}^3)$.  (see
    e.g.~\cite{LiebLoss} for more details).
\end{rem}

\begin{rem}
    The validity of (\textbf{h2}) when $V$ is the Coulomb potential
    \begin{equation}
	\label{eq:coulomb}
	V(x) = -\frac{Ze^{2}}{\abs{x}} \quad \text{\emph{(in cgs units)}}
    \end{equation}
 follows from important inequalities.  Let us recall
    them here.
    \begin{description}
	\item[Hardy] for all $\psi \in H^1(\R^3)$
	\begin{equation*}
	    \| |x|^{-1} \psi \|_{_{L^2}} \leq 2 \| \nabla \psi
	    \|_{_{L^2}} \leq \frac{2}{c} \| \sqrt{-c^{2} \Delta +
	    m^{2}c^{4}} \psi \|_{_{L^2}}
	\end{equation*}

	\item[Kato, Herbst \cite{Herbst_1977}] for all $\psi \in
	H^{1/2}(\R^3) $
	\begin{equation*}
	    \left( \psi, |x|^{-1} \psi \right)_{_{L^2}} \leq
	    \frac{\pi}{2} \left(\psi , \sqrt{-\Delta} \psi
	    \right)_{_{L^2}} \leq \frac{\pi}{2c} \left(\psi ,
	    \sqrt{-c^{2} \Delta + m^{2}c^{4}} \psi
	    \right)_{_{L^2}}
	\end{equation*}

	\item[Tix \cite{Tix_1998}] for all $\psi \in H^{1/2}(\R^3,
	\C^4) $
	\begin{equation*}
	    \left(\Lambda_{+} \psi, |x|^{-1} \Lambda_{+} \psi
	    \right)_{_{L^2}} \leq \frac{1}{2c}{(\frac{\pi}{2} +
	    \frac{2}{\pi})} \left(\Lambda_{+} \psi ,
	    \sqrt{-c^{2} \Delta + m^{2}c^{4}}	    \Lambda_{+} \psi \right)_{_{L^2}}.
	\end{equation*}
    \end{description}
Note that  
    (\textbf{h2}) is satisfied for the electrostatic potential
    provided $0 < Z < 68$ by Hardy, $0 < Z < 87$ by Kato and $0 < Z < Z_{c} = 124$ by Tix's
    inequality.
    
\end{rem}

Many efforts have been devoted to the characterization and computation
of the eigenvalues for the Dirac-Coulomb Hamiltonian (i.e.~the
operator $\free + V$ when $V$ is given by \eqref{eq:coulomb}), see
\cite{Dolbeault_Esteban_Sere_2000} and references therein.  

Due to the unboundedness of the spectrum of the free Dirac operator,
attention has been given also to approximate Hamiltonians constructed
by using projectors.  One of the first attempts in this direction was
made by Brown and Ravenhall \cite{Brown_Ravenhall_1951}.

The Brown-Ravenhall Hamiltonian is defined as
\begin{equation*}
    \mathcal{B} = \Lambda_{+} (D_{0} -\frac{Ze^{2}}{\abs{x}} )
    \Lambda_{+}.
\end{equation*}
This Hamiltonian $\mathcal{B}$ has been considered also in the study
of the ``stability of matter'' for relativistic multi-particle systems
(see \cite{Lieb_Siedentop_Solovej_1997}).

In \cite{Evans_Perry_Siedentop_1996} it is proved that the operator
$\mathcal{B} $ is bounded from below if and only if $Z \leq Z_{c}$.
Then, Tix in \cite{Tix_1998} (see also \cite{Burenkov_Evans_1998})
proved that the operator $ \mathcal{B} $ is strictly positive for $Z
\leq Z_{c}$.

Under our assumptions the quadratic form associate to $\mathcal{B}
=\Lambda_{+} (D_{0} + V) \Lambda_{+}$ is positive definite.  Hence,
by the Friedrichs extension theorem, $\mathcal{B}$ can be defined as a
unique self-adjoint positive operator with domain contained in the
form domain $\mathcal{Q}(|\free|) = H^{1/2}(\mathbb{R}^3,
\mathbb{C}^{4})$.  Moreover, by the KLMN theorem, $\mathcal{B}$ may
also be defined via quadratic forms as a form sum with form domain
$\mathcal{Q}(\mathcal{B}) = H^{1/2}(\mathbb{R}^3, \mathbb{C}^{4})$.
The resulting self-adjoint extensions are equal (see
\cite{Reed_Simon_II_75}).  Hence
\begin{equation*}
    \Lambda_{+} (\free + V) \Lambda_{+} = \Lambda_{+} \free\Lambda_{+}
    + \Lambda_{+} V \Lambda_{+} = \Lambda_{+} \sqrt{- c^{2} 
    \Delta + m^{2}c^{4}} \Lambda_{+} + \Lambda_{+} V \Lambda_{+},
\end{equation*}

\begin{rem}
    \label{rem:assumptions}
    The assumptions (h1)-(h2) and (h3) are very similar to the ones
    given in \cite{Dolbeault_Esteban_Sere_2000}.  Our assumption (h2)
    is slight more restrictive and it allows us to apply the KLMN
    theorem.
\end{rem}

Follows from \eqref{eq:estimate_weakspace} below  that $V$ is a compact
operator from $H^1$ to $H^{-1}$ (but not from $H^{1/2}$ to
$H^{-1/2}$), and this is enough to guarantee that the perturbation
$\Lambda_{+}V\Lambda_{+}$ does not modify the essential spectrum.
Namely, $\sigma_{ess}(\mathcal{B}) = [mc^2, + \infty)$ (see
\cite[Corollary 4 to Weyl's essential spectrum theorem
XIII.14]{Reed_Simon_IV_78}).
 
\begin{notation}
    To simplify the notation we will denote simply with $H^{1/2}$ the
    Hilbert space $H^{1/2}(\mathbb{R}^{3},\mathbb{C}^{2})$, with
    $L^{2}$ the space $L^{2}(\mathbb{R}^{3},\mathbb{C}^{2})$  or $L^{2}(\mathbb{R}^{3},\mathbb{C}^{4})$ as appropriate,  and with
    $H^{1}$ the space $H^{1}(\mathbb{R}^{4}_{+},\mathbb{C}^{2})$ where
    $\mathbb{R}^{4}_{+} = \settc{(x,y_{1},\ldots,y_{3}) \in
    \mathbb{R}^{4}}{x > 0}$.
\end{notation}

In the FW-representation (since $ \fw \Lambda_{+} \fw^{-1} =
\frac{1}{2}(\I_{_4} \pm \mathbf{\beta})$) the associated quadratic
form acting on $\mathcal{H}_{+} $, reduces to $ 2 \times 2
$-(Hermitian) matrix form with domain $\mathcal{Q}(
\mathcal{B}_{_\text{FW}}) = H^{1/2}(\mathbb{R}^3, \mathbb{C}^{2})$ and
for any $ \psi, \phi \in \mathcal{Q}( \mathcal{B}_{_\text{FW}}) $ is
defined by
\begin{equation*}
    (\phi, \mathcal{B}_{_\text{FW}} \psi )_{L^2} = ( \phi,
    \sqrt{-c^{2}\Delta + m^{2}c^{4}} \I_{_2} \psi
    )_{L^2} + ( \phi, V^{^{2 \times
    2}}_{_\text{FW}} \psi )_{L^2}
\end{equation*}
where
\begin{align*}
    &V^{^{2 \times 2}}_{_\text{FW}} \psi = Q^{*}\fw V\fw^{-1} Q \psi, 
    \qquad \psi \in H^{1/2}, \\
    &Q \colon \mathbb{C}^{2} \to \mathbb{C}^{4}, \qquad Q(z_{1},
    z_{2}) = (z_{1}, z_{2}, 0, 0) \\
    &Q^{*} \colon \mathbb{C}^{4} \to \mathbb{C}^{2}, \qquad
    Q^{*}(z_{1}, z_{2}, z_{3}, z_{4}) = (z_{1}, z_{2})
\end{align*}
so that
\begin{align}
    \label{eq:D0inR2}
    ( \phi, \sqrt{-c^{2} \Delta + m^{2}c^{4}} \I_{_2} \psi
    )_{L^2} &= (\Lambda_{+} \fw^{-1} Q \phi ,
    D_{0} \Lambda_{+} \fw^{-1} Q \psi )_{_{L^2(\R^3, \C^{4})}} \\
    &= \left(\Lambda_{+} \fw^{-1}
    \begin{pmatrix} 
	\phi \\ 0
    \end{pmatrix}
    , D_{0} \Lambda_{+} \fw^{-1}
    \begin{pmatrix}
	\psi \\ 0
    \end{pmatrix}
    \right)_{_{L^2(\R^3, \C^{4})}}.\notag
\end{align}
and
\begin{align}
    \label{eq:VinR2}
    ( \phi, V^{^{2 \times 2}}_{_\text{FW}} \psi )_{L^2} &= (\fw^{-1} Q
    \phi , V \fw^{-1} Q \psi )_{_{L^2(\R^3, \C^{4})}} \\
    &= \left(\Lambda_{+}\fw^{-1}
    \begin{pmatrix} 
	\phi \\ 0
    \end{pmatrix}
    , V \Lambda_{+}\fw^{-1}
    \begin{pmatrix}
	\psi \\ 0
    \end{pmatrix}
    \right)_{_{L^2(\R^3, \C^{4})}}.\notag
\end{align}
Note that $ \fw^{-1}Q\varphi = \Lambda_{+} \fw^{-1}Q\varphi \in
\Lambda_{+} L^{2}(\R^3, \C^{4}) $ for any $ \varphi \in L^{2}$.

The operator $\sqrt{-c^{2} \Delta + m^{2}c^{4}}$, exactly as
for the fractional Laplacian, can be related to a Dirichlet to Neumann
operator (see for example \cite{CabreMorales05} for problems involving
the fractional laplacian, and \cite{CZNolasco2011, CZNolasco2013} for
more closely related models).

For any given function $u \in \mathcal{S}(\R^3)$ we  consider the
following Dirichlet boundary problem
\begin{equation*}
    \begin{cases}
	- \partial^{2}_{x}v - c^{2}  \Delta_{y} v + m^{2}
	c^{4} v = 0 & \text{in } \mathbb{R}^{4}_{+} = \settc{(x,y) \in
	\mathbb{R} \times \mathbb{R}^{3}}{x > 0 } \\
	v(0,y) = u(y) &\text{for } y \in \mathbb{R}^{3} = \partial
	\mathbb{R}^{4}_{+}.
    \end{cases}
\end{equation*}
Solving the equation via Fourier transform (w.r.t. $y \in
\R^3$) we get
\begin{equation*}
    v (x,y) = \frac{1}{(2 \pi)^{3/2}} \int_{\R^3} \text{e}^{ i p
    \cdot y} \hat{u}(p) \text{\rm e}^{- \sqrt{ c^{2}|p|^2 + m^2c^{4}} 
    x} \, dp. 
\end{equation*}

Let us define
\begin{equation*}
    \mathcal{T} u (y) = \frac{\partial v}{\partial \nu}(0,y) = -
    \frac{\partial v}{\partial x}(0,y);
\end{equation*}
hence 
\begin{equation*}
    \mathcal{T} u (y) = - \frac{\partial v}{\partial x}(0,y) =
    \frac{1}{(2 \pi)^{3/2}} \int_{\R^3} \text{e}^{ i p \cdot y} \sqrt{
    c^{2}|p|^2 + m^2c^{4}} \, \hat{u}(p) \, dp
\end{equation*}
namely $\mathcal{T} = \sqrt{-c^{2} \Delta_{y} + m^{2}c^{4}}$
on the dense domain $\mathcal{S}(\mathbb{R}^{3})$.

Our aim is to prove a variational characterization of the eigenvalues
and eigenvectors of $\mathcal{B}_{_\text{FW}} $ different from the
classical Rayleigh quotient and which gives rise, as we will see
later, to an alternative eigenvalues problem (see
($\mathcal{E}_{_{k}}$) below) for $\mathcal{B}_{_\text{FW}} $
involving the Dirichlet to Neumann operator.  We believe that such a
characterization can be useful for a finer analysis of the properties
---such as regularity and exponential decay--- of eigenfunctions,
which have been object of investigation with different techniques in
\cite{Bach_Matte_2001}.

We consider the auxiliary functional $\mathcal{I} (\phi)$ defined on
$H^{1}(\R^4_{+},\mathbb{C}^2)$
\begin{equation*}
    \mathcal{I} (\phi) = \iint_{\R^4_{+}} (|\partial_{x} \phi |^2 +
    c^{2}  |\nabla_{y} \phi |^2 + m^2 c^4 |\phi |^2 ) \, dx\,dy
    + \int_{\R^3} ( \phi_{_{tr}}, V^{^{2 \times 2}}_{_\text{FW}}
    \phi_{_{tr}} )_{_{\C^2}} \, dy
\end{equation*}
where $\phi_{_{tr}} \in H^{1/2}$ denotes the
trace of $\phi \in H^{1}$ on $\partial
\mathbb{R}^{4}_{+} = \mathbb{R}^{3}$.
 
We have the following result.
 
\begin{thm}
    \label{thm:eigenvalues}
    Let (h1)-(h2)-(h3) hold.  Then there exist $\lambda_{1} \leq
    \lambda_{2} \leq \ldots \leq \lambda_{k} \leq \ldots$ and
    $\phi_{1}, \phi_{2}, \ldots, \phi_{k}, \ldots \in
    H^{1}(\R^4_{+},\mathbb{C}^2)$ such that, for all $k \in
    \mathbb{N}$
    \begin{equation*}
	\lambda_{k} = \mathcal{I} (\phi_{k}) = \inf_{X_{k} }
	\mathcal{I} (\phi)
    \end{equation*}
    where 
    \begin{equation*}
	X_{1} = \settc{\phi \in H^{1}}{\abs{\phi_{\text{tr}}}_{L^{2}}
	= 1} .
    \end{equation*}
    and,  for $ 1 < k \in \mathbb{N}$
    \begin{equation*}
	X_{k} = \settc{\phi \in H^{1}}{\abs{\phi_{\text{tr}}}_{L^{2}}
	= 1, \ (\phi_{\text{tr}},  (\phi_i)_{tr})_{L^{2}} = 0, \ i = 1,
	\ldots, k-1} .
    \end{equation*}

    Moreover $ \{ \lambda_{k} \}_{_{k \geq 1}} \in
    \sigma_{\text{disc}}(\mathcal{B}_{_\text{FW}}) =
    \sigma_{\text{disc}}(\mathcal{B})$ and
    \begin{equation*}
	0 < \lambda_1 \leq \ldots \leq \lambda_k \leq \lambda_{k + 1}
	\, \to \, \inf \{ \sigma_{\text{ess}}
	(\mathcal{B}_{_\text{FW}}) \} = mc^2 \quad \text{for} \, \, k
	\to + \infty .
    \end{equation*}
    The corresponding eigenfunctions are $\varphi_k = (\phi_k)_{tr}
    \in H^{1/2}(\R^3, \mathbb{C}^{2})$, and  $\phi_k \in
    H^{1}(\R^4_{+},\mathbb{C}^2)$ are weak solution of the Neumann
    problem
    \begin{equation*}
	\tag{$\mathcal{E}_{_{k}}$}
	\begin{cases}
	    -\partial^2_x \phi_k -  c^2 \Delta_{y} \phi_k +
	    m^{2}c^4 \phi_k= 0 & \quad \text{in } \mathbb{R}^{4}_{+}
	    \\
	    \displaystyle{ \frac{\partial \phi_k}{\partial \nu} } +
	    V^{^{2 \times 2}}_{_\text{FW}} \varphi_k = \lambda_k
	    \varphi_k & \quad \text{on } \partial\mathbb{R}^{4}_{+} =
	    \mathbb{R}^{3}.
	\end{cases}
    \end{equation*}
\end{thm}

\section{Proof of Theorem \ref{thm:eigenvalues}}

It is convenient to introduce the following (equivalent) norm in
$H^{1}(\R^4_{+},\mathbb{C}^2)$
\begin{equation*}
    \|\phi \|^2_{_{H^1}} = \iint_{\R^4_{+}}
    (|\partial_{x} \phi |^2 + c^{2}  |\nabla_{y}
    \phi |^2+ m^2 c^4 |\phi |^2) \, dx\,dy .
\end{equation*}

The following property   can be easily verified
\begin{lem}
    \label{lem:betterminimizer}
    For $w \in H^{1}(\mathbb{R}^{4}_{+})$, let $u = w_{\text{tr}} \in
    H^{1/2}(\mathbb{R}^{3})$ be the trace of $w$, $\hat{u} =
    \mathcal{F}(u)$ and
    \begin{equation*}
	v(x,y) = \mathcal{F}^{-1}_{y}\bigl[\hat{u}(p ) e^{-\sqrt{
	c^{2}|p|^2 + m^2c^{4}}x}\bigr].
    \end{equation*}
    
    Then $v \in H^{1}(\mathbb{R}^{4}_{+})$, $\| v 
    \|_{H^1(\mathbb{R}^{4})} = \|u
    \|_{H^{1/2}(\mathbb{R}^{3})}$, and
    \begin{align*}
	\int_{\R^3} \sqrt{ c^{2}|p|^2 + m^2c^{4}} \,|\hat{u}|^2 \, dp
	& = \iint_{\R^4_{+}} (|\partial_{x} v |^2 + c^{2} 
	|\nabla_{y} v |^2 + m^2 c^4 |v |^2 )\, dx\,dy \\
	& \leq \iint_{\R^4_{+}} (|\partial_{x} w |^2 + c^{2} 
	|\nabla_{y} w |^2 + m^2 c^4 |w |^2 )\, dx\,dy.
    \end{align*}
\end{lem}

Let introduce also  the following  norm in the weak $L^q$-space:
\begin{equation*}
    | f |_{_{L^{q}_{w}}} = \sup \settc{ |A| ^{- 1/r} \int_{A} |f(x)| \, dx
    }{A \subset \mathbb{R}^{3}, \text{measurable},\ 0<\abs{A}
    <+\infty}
\end{equation*}
where $1/q + 1/r = 1$.

We have  the following fact:

\begin{lem}
    \label{lem:estimate_weakspace}
    Let $V \in L_{ w}^{3}(\R^3)$ and $f \in H^{1/2} (\R^3)$.  
    
    We have that
    \begin{equation}
	\label{eq:estimate_weakspace}
	|V^{1/2} f |_{_{L^2}} \leq C | V |^{1/2}_{_{L^3_{w}}} | f
	|_{_{H^{1/2}}}.
    \end{equation}

\end{lem}

\begin{proof}
    Follows from~\cite[(42)]{Lenzmann07} that the Green function
    $G^{\mu}_{\alpha}$ of $(- \Delta + \mu^2 )^{\alpha/2}$ belongs to
    $L_{ w}^{3/(3 - \alpha)}(\R^3)$ if $\mu \geq 0$ and $0 < \alpha <
    3$.
    
    Then, given $f \in H^{1/2} (\R^3)$, let $h= (- \Delta +
    \mu^2)^{1/4} f \in L^2 (\R^3)$, $f = G^{\mu}_{1/2} \ast h$.  From
    the weak Young's inequality (see Proposition \ref{prop: Kovalenko}
    in appendix \ref{sec:appendixA}), we deduce
    \begin{align*}
	|V^{1/2} f |_{_{L^2}} &= |V^{1/2} ( G^{\mu}_{1/2} \ast h)
	|_{_{L^2}} \leq C | V^{1/2} |_{_{L^6_{w}}} | G^{\mu}_{1/2}
	|_{_{L^{6/5} _{w}}} |h |_{_{L^2}}\\
	&\leq C | V |_{_{L^3_{w}}} |(- \Delta + \mu^2)^{1/4} f
	|_{_{L^2}} \leq C | V |^{1/2}_{_{L^3_{w}}} | f |_{_{H^{1/2}}}.
    \end{align*}
\end{proof}

   Now, we introduce   the differential  $  d \mathcal{I} (\phi) : H^1 \to \R$ of the quadratic functional  $\mathcal{I}$ \begin{align*}
   d \mathcal{I} (\phi) [h]  =  & 2 \RE  \iint_{\R^4_{+}} ( \left(\partial_{x}\phi, h \right)_{_{\C^2}} + 
   c^2  \left(\nabla_{y}\phi,   \nabla_{y} h  \right)_{_{\C^2}} +  m^2 c^4 \left(\phi, h \right)_{_{\C^2}} ) \\
 & \qquad \qquad   + 2 \RE  (\varphi,V^{2 \times 2}_{_\text{FW}} h)_{_{L^2}} 
  \end{align*}
and  also we  compute   $d \mathcal{I}(\phi)[\chi^{2}\phi]  $ for any $\chi = \chi(y) \in
C^{\infty}_{0}(\mathbb{R}^{3})$ and $\phi \in
H^{1}$.  We have 
\begin{align*}
  \RE \iint_{\R^4_{+}} \left(\partial_{x}\phi, \partial_{x}(\chi^{2}\phi)  \right)_{_{\C^2}} &=  \iint_{\R^4_{+}}  \abs{\partial_{x}(\chi\phi)}^{2} \\
   \RE \iint_{\R^4_{+}} \left(\partial_{x}\phi, \partial_{x}(\chi^{2}\phi)  \right)_{_{\C^2}} &=  \iint_{\R^4_{+}}  \abs{\partial_{x}(\chi\phi)}^{2} \\
\end{align*}
and we compute, adding and substracting $\abs{\nabla_{y}(\chi\phi)}^{2} $,
\begin{align*}
    \RE  \iint_{\R^4_{+}}  & \left(\nabla_{y}\phi,   \nabla_{y}(\chi^{2}\phi)  \right)_{_{\C^2}}  =  
     \iint_{\R^4_{+}} 
     \abs{\nabla_{y}(\chi\phi)}^{2} \\
     & \quad      -     \iint_{\R^4_{+}} ( \chi^{2}\abs{\nabla_{y}\phi}^{2} +
     \abs{\phi}^{2}\abs{\nabla_{y}\chi}^{2} + 2
     \RE \left(\phi \nabla_{y}\chi ,   \chi  \nabla_{y}\phi  \right)_{_{\C^2}} )  \\
     &  \qquad  \qquad  +  2 \RE  \iint_{\R^4_{+}} (\left(  \nabla_{y}\phi,  \chi    \phi  \nabla_{y} \chi 
    \right)_{_{\C^2}}  + 
    \chi^{2}\abs{\nabla_{y}\phi}^{2} )\\
    & \qquad \qquad \qquad = \iint_{\R^4_{+}}
    ( \abs{\nabla_{y}(\chi\phi)}^{2} - 
    \abs{\phi}^{2}\abs{\nabla_{y}\chi}^{2} ).
 \end{align*}
and, letting $\varphi = \phi_{tr} \in H^{1/2}$,
\begin{equation*}
  \RE  (\varphi,V^{2 \times 2}_{_\text{FW}} \chi^{2}\varphi)_{_{L^2}} 
    = (\chi\varphi,V^{2 \times 2}_{_\text{FW}} \chi\varphi)_{_{L^2}} + \RE
    (\varphi,[V^{2 \times 2}_{_\text{FW}}, \chi] \chi\varphi)_{_{L^2}}
\end{equation*}
where $[ \, \cdot \, , \, \cdot \,]$ denotes the commutator of
    operators. 
Then, we have that
\begin{equation*}
    d \mathcal{I}(\phi)[\chi^{2}\phi] = d
    \mathcal{I}(\chi\phi)[\chi\phi] - 2c^{2}  \iint_{\R^4_{+}} |\nabla_{y}\chi|^2 |\phi|^2
    + 2 \RE
    (\varphi,[V^{2 \times 2}_{_\text{FW}}, \chi] \chi\varphi)_{_{L^2}}.
\end{equation*}

Now, we use the commutator identity
\begin{equation*}
    [ABC, D] = AB[C,D] + A[B,D]C + [A,D]BC
\end{equation*}
and the fact that $[V,\chi] = 0$ to deduce that (we let $Q\varphi = 
\fw \psi$)
\begin{align*}
    (\varphi,[V^{2 \times 2}_{_\text{FW}}, & \chi]  \chi\varphi)_{_{L^2}} =
    (\varphi,Q^{*}[\fw V \fw^{-1}, \chi] Q\chi\varphi)_{_{L^2}}  =   (Q \varphi,[\fw V \fw^{-1}, \chi] \chi Q\varphi)_{_{L^2}}  \\
   &= (Q \varphi, \fw V [\fw^{-1}, \chi] \chi Q\varphi)_{_{L^2}} + (Q\varphi,
    [\fw, \chi] V \fw^{-1} \chi Q\varphi)_{_{L^2}} \\
    &= (\fw \psi, \fw V [\fw^{-1}, \chi] \chi \fw \psi)_{_{L^2}} + (\fw \psi,
    [\fw, \chi] V \fw^{-1} \chi \fw \psi)_{_{L^2}} \\
    &= (V \psi, [\fw^{-1}, \chi] \chi \fw \psi)_{_{L^2}} + ([\chi, \fw^{-1}]
    \fw \psi, V \fw^{-1} \chi \fw \psi)_{_{L^2}}.
\end{align*}
In addition
\begin{equation*}
 d\mathcal{I}(\chi\phi)[\chi\phi] =  2 \mathcal{I}(\chi\phi). 
\end{equation*}
Finally we get
\begin{multline}
  \label{eq:differenziale_commutato}
    d \mathcal{I}(\phi)[\chi^{2}\phi] = 
  2  \mathcal{I}(\chi\phi)  - 2c^{2}   \iint_{\R^4_{+}} |\nabla_{y}\chi|^2 |\phi|^2 
    + 2 \RE (V \psi, [\fw^{-1}, \chi] \chi \fw \psi)_{_{L^2}} \\
  + 2 \RE ([\chi,
    \fw^{-1}] \fw \psi, V \fw^{-1} \chi \fw \psi)_{_{L^2}}\\
\end{multline}

We divide the proof of Theorem \ref{thm:eigenvalues} in several
steps.  Let us begin with the existence of the ground state.
  
We consider the following minimization problem :
\begin{equation}
 \label{eq:p1}
     \tag{$\mathcal{P}_{_{1}}$} \lambda_{1} = \inf_{_{ \phi \in S} }
    \mathcal{I} (\phi).
\end{equation}
where $S = \settc{\phi \in
H^{1}}{\abs{\phi_{tr}}_{L^{2}}^{2} = 1}$.

\begin{lem}
    \label{lem:coercive}
    The following holds:
    \begin{itemize}

	\item[(i)] $\mathcal{I} (\phi)$ is bounded by below and
	coercive on $H^{1}$,
	\item[(ii)] $ 0 < \lambda_{1} < mc^2$.
 
    \end{itemize}
 
\end{lem}
 
\begin{proof}
    (i) Let $\phi \in H^{1}$, $\varphi =
    \phi_{tr}$ and define $ \psi= \fw^{-1} Q \varphi$, then
    $\Lambda_{+} \psi = \psi$, hence by (\textbf{h2}),
    \eqref{eq:D0inR2}, \eqref{eq:VinR2} and lemma
    \ref{lem:betterminimizer}, there exists $a \in (0,1)$ such that
    \begin{equation*}
	\begin{split}
	    ( \varphi, V^{^{2 \times 2}}_{_\text{FW}}
	    \varphi )_{L^2} &=
	    (\Lambda_{+} \psi, V \Lambda_{+} \psi )_{_{L^2}} \geq - a (\Lambda_{+} \psi, D_{0}
	    \Lambda_{+} \psi )_{_{L^2}} \\
	    &= - a ( \varphi, \sqrt{-c^{2} \Delta +
	    m^{2}c^{4}}  \I_{2}\, \varphi )_{_{L^2}} \\
	    & \geq {- a} \iint_{\R^4_{+}} (|\partial_{x} \phi
	    |^2+ c^2|\nabla_{y} \phi
	    |^2 + m^2 c^4 |\phi |^2 ) \,
	    dx\,dy
	\end{split}
    \end{equation*}
    Therefore, we may conclude that there exists $\delta >0$ such that
    $\mathcal{I} (\phi) \geq \delta \| \phi
    \|^2_{_{H^{1}}} $.

    (ii) From (i) immediately follows that $ \lambda_1 > 0$.  Now take
    $\phi (x,y) = \text{e}^{- mc^2 x} \varphi(y)$, with $\varphi \in
    C^{\infty}_{0}(\R^{3}, \mathbb{C}^{2} )$, and
    $ |\varphi |_{_{L^2}}=1$, we have
    \begin{equation*}
	\mathcal{I} (\phi) - mc^2 = \frac{1}{2 m} \int_{\R^3}
	|\nabla \varphi|^2 + \int_{\R^3} ( \varphi, V^{^{2
	\times 2}}_{_\text{FW}} \varphi)_{_{\C^2}} =
	\mathcal{E}(\varphi)
    \end{equation*}

    Take, now $\varphi_{_{\eta}} (y) = \eta^{^{3/2}} \varphi(\eta y)$, we
    have $|\varphi_{_{\eta}} |_{_{L^2}}=1$, for any $\eta >0$ and
    setting $\phi_{_{\eta}}(x,y) = \text{e}^{- mc^2 x}\varphi_{_{\eta}}
    (y) $
    \begin{align*}
	\lambda_1 - mc^2 \leq & \inf_{_{\eta >0} } \mathcal{I}
	(\phi_{_{\eta}}) - mc^2 = \inf_{_{\eta >0}
	}\mathcal{E}(\varphi_{_{\eta}} ) = \\
	= & \inf_{_{\eta >0} } \left({\eta ^2} \frac{1}{2 m}
	\int_{\R^3} |\nabla \varphi |^2+ \int_{\R^3} \left(
	\varphi_{_{\eta}}, V^{^{2 \times 2}}_{_\text{FW}}
	\varphi_{_{\eta}} \right)_{_{\C^2}} \right)
    \end{align*}
    We claim that by (h3), 
    \begin{equation*} 
	\limsup _{\eta \to 0^{^{+}}} \frac{1}{{\eta^2}} \left(
	\varphi_{_{\eta}}, V^{^{2 \times 2}}_{_\text{FW}}
	\varphi_{_{\eta}} \right)_{L^2} =
	-\infty,
    \end{equation*}
    which implies that $\lambda_1 - mc^2 <0$.
    
      Indeed, denoting $ \hat{f} (p)= \mathcal{F}(f)(p)$ and
    $U_{_{\text{FW},\eta}} = \mathcal{F}^{-1}U(\eta p)\mathcal{F}$ we have
     \begin{align*}
	& ( \varphi_{_{\eta}}, V^{^{2 \times 2}}_{_\text{FW}}
	\varphi_{_{\eta}} )_{_{L^2}} = (Q\varphi_{_{\eta}}, \fw V(y)
	\fw^{-1} Q \varphi_{_{\eta}} )_{_{L^2}}\\
	&  = (\mathcal{F}Q\varphi_{_{\eta}}, \mathcal{F} \fw V(y) \fw^{-1} Q
	\varphi_{_{\eta}} )_{_{L^2}} = (Q \hat{\varphi}_{_{\eta}},
	U(p)\mathcal{F}V(y)\mathcal{F}^{-1}U(p)^{-1} Q \hat{\varphi}_{_{\eta}})_{_{L^2}} \\
	&= (Q\hat{\varphi}, \eta^{3} U(\eta p) \hat{V}(\eta p) * U(\eta
	p)^{-1} Q \hat{\varphi} )_{_{L^2}} = (Q\varphi, U_{_{\text{FW},\eta}} V(\eta^{-1} y)
	U_{_{\text{FW},\eta}}^{-1} Q \varphi)_{_{L^2}}\\
		&= (\varphi, V(\eta^{-1} y) \varphi)_{_{L^2}} + (Q\varphi,
	[U_{_{\text{FW},\eta}}, V(\eta^{-1} y)] U_{_{\text{FW},\eta}}^{-1} Q
	\varphi)_{_{L^2}} .\\
	 \end{align*}
	  Hence, since by  (\textbf{h1}) and  lemma  \ref{lem:estimate_weakspace},  for any $f
    \in H^{1/2}$ we have
      \begin{equation}
	\label{eq:rescale_weak}
	|V^{1/2} (\eta^{-1} y) f |_{_{L^2}} \leq C \eta^{1/2} | V
	|^{1/2}_{_{L^3_{w}}} | f |_{_{H^{1/2}}}.
    \end{equation}
      We get 
    \begin{align*}  
  | (Q\varphi,
	[U_{_{\text{FW},\eta}} , & V(\eta^{-1} y)] 
	U_{_{\text{FW},\eta}}^{-1} Q \varphi)_{_{L^2}} |   \\
	& = \abs{(Q\varphi\,,\,[U_{_\text{FW},\eta} - \I_{4}, V(\eta^{-1}
	y)] U_{_\text{FW},\eta}^{-1} Q \varphi)_{_{L^2}}} \\
	&\leq \abs{V^{1/2}(\eta^{-1} y) (U^{-1}_{_\text{FW},\eta} -
	\I_{4}) Q \varphi}_{_{L^2}} \abs{V^{1/2}(\eta^{-1} y)
	U^{-1}_{_\text{FW},\eta} Q \varphi}_{_{L^2}} \\
	&\quad + \abs{ V^{1/2}(\eta^{-1} y) Q \varphi}_{_{L^2}} \abs{
	V^{1/2}(\eta^{-1} y) (U_{_\text{FW},\eta} - \I_{4})
	U^{-1}_{_\text{FW},\eta} Q \varphi }_{_{L^2}}\\
	& \leq C \eta | V |_{_{L^3_{w}}} \abs{ (U^{-1}_{_\text{FW},\eta}
	- \I_{4}) Q \varphi}_{_{H^{1/2}}} \abs{\varphi}_{_{H^{1/2}}} .
    \end{align*}

    Then recalling that
    \begin{equation*}  
	U^{-1}(\eta p) = a_{+} (\eta p) \I_{4} - a_{-} (\eta p)
	\mathbf{\beta} \frac{\underline{\mathbf{\alpha}} \cdot p}{|p|}
    \end{equation*}
    with $a_{\pm}(\eta p) = \sqrt{\frac{1}{2}( 1 \pm mc^2 / \lambda(\eta
    p))} $, and $\lambda(\eta p)= \sqrt{ \eta^2 c^2 |p|^2 + m^2 c^4 } $,
    we have
    \begin{equation*}  
	U^{-1}(\eta p) - \I_{4} = ( a_{+} (\eta p) - 1)\I_{4} - a_{-}
	(\eta p) \mathbf{\beta} \frac{\underline{\mathbf{\alpha}} \cdot
	p}{|p|}
    \end{equation*}
    and we estimate 
    \begin{equation*}
	|a_{+} (\eta p) - 1| \leq \frac{|mc^2 -  \lambda(\eta p)|}{2 \lambda(\eta p)} \leq 
	\frac{\eta^2 |p|^2}{2 m^2 c^2} 
	  \end{equation*}
	and
	   \begin{equation*}
 | a_{-} (\eta p) | \leq   \left(\frac{\lambda(\eta p) - mc^2 }{2 \lambda(\eta p)}   \right)^{1/2}   
  \leq \frac{\eta |p|}{\sqrt{2} m c} .
      \end{equation*}
        
    Therefore we may conclude 
    \begin{equation*}
	\sup_{p \in \R^3 } \frac{ |a_{+} (\eta p) - 1| }{(1 + |p|)^2}
	\leq C \eta ^2 \qquad \text{and} \quad \sup_{p \in \R^3 }
	\frac{ |a_{-} (\eta p) | }{1 + |p|} \leq C \eta
    \end{equation*}
for some constant $C>0$.

    Moreover, since $ \hat{\varphi}$ is the Fourier transform of a
    compact support $C^{\infty}$-function, it
    decays at infinity faster than any power, namely for any $\alpha
    >0$ there exists a positive constant $C_{\alpha} >0$ such that
    \begin{equation}
	\label{eq:decay1}
	| \hat{\varphi}(p)| \leq \frac{C_{\alpha} }{(1 +
	|p|)^{\alpha}}.
    \end{equation}

   Then,  we have 
    \begin{align*}  
	\abs{(U^{-1}_{_\text{FW},\eta} & - \I_{4}) Q
	\varphi}^2_{_{H^{1/2}}} = \int_{\R^3} (1 + |p|)
	\abs{\mathcal{F} ((U^{-1}_{_\text{FW},\eta} - \I_{4}) Q
	\varphi}^2 \, dp \\
	& \leq 2 \int_{\R^3} (1 + |p|) |( a_{+} (\eta p) - 1)
	\hat{\varphi} |^2 + 2 \int_{\R^3} (1 + |p|) \left|a_{-} (\eta p)
	\mathbf{\beta} \frac{\underline{\mathbf{\alpha}} \cdot p}{|p|}
	Q \hat{\varphi}
	\right|^2 \, dp \\
	& \leq 2 \int_{\R^3} (1 + |p|) | a_{+} (\eta p) - 1|^2
	|\hat{\varphi} |^2 \, dp + 2 \int_{\R^3} (1 + |p|) |a_{-} (\eta p) |^2 |
	\hat{\varphi }|^2 \, dp \\
	& \leq 2 \sup_{p \in \R^3 } \frac{ |a_{+} (\eta p) - 1|^2 }{(1
	+ |p|)^4} \int_{\R^3} (1 + |p|)^5 |\hat{\varphi} |^2 \, dp  \\
	&\qquad + 2 \sup_{p \in \R^3 } \frac{ |a_{-} (\eta p) |^2 }{(1
	+ |p|)^2} \int_{\R^3} (1 + |p|) ^3 | \hat{\varphi }|^2 \, dp  \leq  C \eta^2 .
    \end{align*} 
    for some constant $C >0$  depending only on $\varphi$ and $\eta$ sufficiently small. 
    We get 
    \begin{equation*}
	( \varphi_{_{\eta}}, V^{^{2 \times 2}}_{_\text{FW}}
	\varphi_{_{\eta}} )_{_{L^2}} = ( \varphi, V(\eta^{-1} y) \varphi
	)_{_{L^2}} + O(\eta^2).
    \end{equation*}
    By   (\textbf{h3})  for any $K >0$ there exists $R >0$ such that for any $|y|
    >R$ we have $ V(y) \leq - K / |y|^2 $ a.e..  We have
    \begin{align*}
	( \varphi, V(\eta^{-1} y) \varphi )_{_{L^2}} &=
	\int_{\{\eta^{-1} |y | \leq R \}} V(\eta^{-1} y) | \varphi
	|^2 + \int_{\{ \eta^{-1} |y | > R \}} V(\eta^{-1} y) |
	\varphi |^2\\
	&\leq \eta^3 \sup_{|y| \leq \eta R} | \varphi(y) |^2
	\int_{\{ |y| \leq R \}} |V(y) | - K \eta^2 \int_{\{ |y | > \eta
	R \}} \frac{ 1}{|y|^2} | \varphi |^2\\
	&\leq C( \eta^3 - K \eta^2 )
    \end{align*}
    where the constant $C >0 $ depends on $\varphi$ and $R$, and $K >
    0$ is arbitrarily large.
           
    As claimed above we may then conclude that given $\varphi \in
    C^{\infty}_{0}(\R^{3}; \C^2)$
    \begin{equation*}
	\limsup _{\eta \to 0^{^{+}}} \frac{1}{\eta^2} (\varphi_{_{\eta}},
	V^{^{2 \times 2}}_{_\text{FW}} \varphi_{_{\eta}})_{_{L^2}} =
	\limsup _{\eta \to 0^{^{+}}} \frac{1}{\eta^2} ( \varphi,
	V(\eta^{-1} y) \varphi)_{_{L^2}} + \frac{O(\eta^2)}{\eta^2}= -
	\infty
    \end{equation*}
\end{proof}

We will minimize $\mathcal{I}$ on the set 
\begin{equation*}
    S = \settc{\phi \in H^{1}}{\mathcal{G}(\phi) =
    \abs{\phi_{tr}}_{_{L^{2}}}^{2} = 1}.
\end{equation*}
We recall that the tangent space at $S$ at the point $\phi \in S$ is
the set
\begin{equation*}
    T_{\phi}S = \settc{h \in H^{1}}{d\mathcal{G}(\phi)[h] =
    2 \RE(\phi_{tr},h_{tr})_{_{L^{2}} }= 0}
\end{equation*} 
and that $\nabla_{S}\mathcal{I}(\phi)$, the projection of the gradient
on the tangent space $T_{\phi}S$ to $S$ at the point $\phi$ is given
by
\begin{equation*}
    \nabla_{S} \mathcal{I}(\phi) = \nabla \mathcal{I}(\phi) - \mu(\phi) 
    \nabla \mathcal{G}(\phi)
\end{equation*}
where $\nabla \mathcal{I}(\phi) \in H^{1}$ is such that
\begin{equation*}
    (\nabla \mathcal{I}(\phi), h)_{_{H^{1}}} = d\mathcal{I}(\phi)[h] = 2
    \RE (\phi,h)_{_{H^{1}}} + 2 \RE (\phi_{tr},\Vfw h_{tr})_{_{L^{2}}}
    \qquad \text{for all } h \in H^{1},
\end{equation*}
$\nabla \mathcal{G}(\phi) \in H^{1}$ is such that
\begin{equation*}
    (\nabla \mathcal{G}(\phi),h)_{_{H^{1}}} = 2 \RE
    (\phi_{tr},h_{tr})_{_{L^{2}}} \qquad \text{for all } h \in H^{1},
\end{equation*}
and $\mu(\phi)  \in \mathbb{R}$ is such that $\nabla_{S}\mathcal{I}(\phi) \in
T_{\phi}S$.  Then
\begin{equation*}
    0 = (\nabla \mathcal{G}(\phi), \nabla_{S}
    \mathcal{I}(\phi))_{_{H^{1}}} = (\nabla \mathcal{G}(\phi), \nabla
    \mathcal{I}(\phi))_{_{H^{1}}} - \mu (\phi) \norm{\nabla
    \mathcal{G}(\phi)}^{2}_{_{H^{1}}} 
\end{equation*}
and
\begin{equation*}
    \mu (\phi) = \frac{(\nabla \mathcal{G}(\phi), \nabla_{S}
    \mathcal{I}(\phi))_{_{H^{1}}}}{\norm{\nabla
    \mathcal{G}(\phi)}^{2}_{_{H^{1}}} }
\end{equation*}

From
\begin{align*}
    (\nabla_{S} \mathcal{I}(\phi),\phi)_{_{H^{1}}} &= (\nabla
    \mathcal{I}(\phi),\phi)_{_{H^{1}}} - \mu(\phi)  (\nabla
    \mathcal{G}(\phi),\phi)_{_{H^{1}}} \\
    &= 2\mathcal{I}(\phi) - 2\mu (\phi) \mathcal{G}(\phi) = 2 
    \mathcal{I}(\phi) - 2\mu(\phi) 
\end{align*}
we also deduce that
\begin{equation}
    \label{eq:stime_moltiplicatore}
    \mu(\phi)  = \mathcal{I}(\phi) -  \frac{1}{2} (\nabla_{S}
    \mathcal{I}(\phi),\phi)_{_{H^{1}}}
\end{equation}

We now recall the following well known result
\begin{lem}
    There exists a Palais-Smale minimizing sequence $\phi_{n}$ for
    $\mathcal{I}$ on the set $S =
    \settc{\phi}{\abs{\phi_{tr}}_{L^{2}}^{2} = 1}$, that is a
    sequence such that, denoting $\varphi_{n} = (\phi_{n})_{tr}$,
    \begin{equation*}
	\mathcal{I}(\phi_{n}) \to \lambda_{1}, \qquad
	\nabla_{S}\mathcal{I}(\phi_{n}) \to 0, \qquad
	\abs{\varphi_{n}}_{L^{2}}^{2} = 1
    \end{equation*}
\end{lem}

\begin{proof}
    Assuming that the result does not hold, one deduces that there
    exist $\epsilon > 0$, $\delta > 0$ such that
    $\norm{\nabla_{S}\mathcal{I}(\phi)} \geq \delta > 0$ for all $\phi
    \in S$ such that $\lambda_{1}-\epsilon < \mathcal{I}(\phi) <
    \lambda_{1} + \epsilon$.  Building a gradient flow $\eta' =
    \nabla_{S}\mathcal{I}(\eta)$, which leaves $S$ invariant and
    pushes $\{\mathcal{I} < \lambda_{1} + \epsilon\} \cap S$ into
    $\{\mathcal{I} < \lambda_{1} - \epsilon\} \cap S$, one easily
    reaches a contradiction.
    
    One can  prove the lemma also using Ekeland's variational
    principle.
\end{proof}

\begin{lem}
    \label{lem:potential_term}
    Let $\phi_n $ be a Palais Smale sequence at some level $\lambda
    \geq 0$ for $\mathcal{I}$ on $S$.  Let $\varphi_{n} = (\phi_{n})_{tr}$.
            
    If $\varphi_{n} \rightharpoonup 0$ in $H^{1/2}$ then
    \begin{equation*}
	( \varphi_n, V^{^{2 \times 2}}_{_\text{FW}} \varphi_n
	)_{_{L^2}} \to 0.
    \end{equation*} 
\end{lem}
  
\begin{proof}
   
    Since $ \mathcal{I}$ is coercive, $ \phi_n $ is bounded
    $H^{1}$, $ \varphi_n $ is bounded in
    $H^{1/2}$ and, by Sobolev embedding,
    relatively compact in $L^p_{\text{loc}}$ for $p \in [2, 3)$.  From
    \eqref{eq:stime_moltiplicatore} follows that also $\mu_{n}$ is
    bounded.
    
    By {\bf(h3)} $V \in L^{\infty}(\R^3 \setminus \overline{B}_{R_0})$ and
    for any $\varepsilon >0$, the set $A_{\varepsilon} = \settc{ y \in \R^3
    \setminus \overline{B}_{R_0}}{|V(y) | \geq \varepsilon} $ is bounded.

    Take a  radial function  $\chi \in C_0^{\infty}(\R^3)$, with values in 
    $[0,1]$ such that $\chi(y) = 1$ for  $y \in B_{1}$ and $\chi(y)= 0$ for 
    $y \in \R^{3} \setminus B_{ 2}$ and 
    let $\chi_{_{R}}(y) = \chi (R^{-1} y)$.

    \subsection*{Step 1: }    $\abs{(\varphi_n, V^{^{2 \times 2}}_{_\text{FW}}
    \varphi_n )_{_{L^2}} - (\chi_{_{R}} \varphi_n,
    V^{^{2 \times 2}}_{_\text{FW}} \chi_{_{R}} \varphi_n )_{_{L^2}}} \to 0$ as $R \to +\infty$.

    Let $\psi_n =\fw^{-1}Q\varphi_n$, we have that $\psi_n
    \rightharpoonup 0$ in $H^{1/2}$ and also 
    that
    \begin{equation*}
	\label{eq:decomposizione_potenziale}
	(\varphi_n, V^{^{2 \times 2}}_{_\text{FW}} \varphi_n
	)_{L^2} = ( \psi_n, V \psi_n
	)_{_{L^2}}
	= ( \psi_n, (1-\chi_{_{R}}^{2})V \psi_n )_{_{L^2}}+ ( \psi_n, \chi_{_{R}}^{2}V \psi_n
	)_{_{L^2}}.
    \end{equation*} 
    
    Taking $R > R_0$ in such a way that $A_{\varepsilon} \subset
    B_{R}$ we have
    \begin{equation*}
	\abs{( \psi_n, (1-\chi_{_{R}}^{2})V \psi_n )_{_{L^2}}} \leq \epsilon \abs{\psi_n}^{2}_{_{L^2}} \leq C \varepsilon. 
    \end{equation*}

    On the other hand, (note that {$\Lambda_{-} ( \chi_{_{R}} \psi_n )
    \not = 0 $})
    \begin{align*}
	( \psi_n, \chi^2_{_{R}}  & V \psi_n )_{_{L^2}}  = (\chi_{_{R}}  \fw^{-1}Q\varphi_n, V
	\chi_{_{R}}  \fw^{-1}Q\varphi_n )_{_{L^2}} \\
	&= ( \chi_{_{R}} \varphi_n, Q^{*} \fw V \chi_{_{R}}\fw^{-1}Q\varphi_n
	)_{_{L^2}}  \\
	& \qquad + ( [\chi_{_{R}}, \fw^{-1}] Q \varphi_n, V
	\chi_{_{R}}\fw^{-1}Q\varphi_n )_{_{L^2}}\\
	&= ( \chi_{_{R}}\varphi_n, V^{2 \times 2}_{_{\text{FW}}} \chi_{_{R}}
	\varphi_n )_{_{L^2}} + ( \chi_{_{R}}\varphi_n, Q^{*} \fw V
	[\chi_{_{R}}, \fw^{-1}]Q \varphi_n )_{_{L^2}} \\
	&\qquad + ( [\chi_{_{R}}, \fw^{-1}] \fw \psi_n, V
	\chi_{_{R}}\psi_n )_{_{L^2}} \\
	&= ( \chi_{_{R}}\varphi_n, V^{2 \times 2}_{_{\text{FW}}} \chi_{_{R}}
	\varphi_n )_{_{L^2}} + (\fw^{-1} \chi_{_{R}}\fw \psi_n, V
	[\chi_{_{R}}, \fw^{-1}] \fw \psi_n )_{_{L^2}} \\
	&\qquad + ( [\chi_{_{R}}, \fw^{-1}] \fw \psi_n, V \chi_{_{R}}
	\psi_n )_{_{L^2}} \\
	&= ( \chi_{_{R}}\varphi_n, V^{2 \times 2}_{_{\text{FW}}} \chi_{_{R}}
	\varphi_n )_{_{L^2}} + ([\fw^{-1}, \chi_{_{R}}]\fw \psi_n, V
	[\chi_{_{R}}, \fw^{-1}]\fw \psi_n)_{_{L^{2}}} \\
	&\qquad + (\chi_{_{R}} \psi_n, V [\chi_{_{R}}, \fw^{-1}] \fw \psi_n
	)_{_{L^2}} + ( [\chi_{_{R}}, \fw^{-1}] \fw \psi_n, V
	\chi_{_{R}}\psi_n )_{_{L^2}} \\
	&= ( \chi_{_{R}}\varphi_n, V^{2 \times 2}_{_{\text{FW}}} \chi_{_{R}}
	\varphi_n )_{_{L^2}} -   | V^{1/2} [\fw^{-1}, \chi_{_{R}}]\fw \psi_n|^2_{_{L^{2}}} \\
	&\qquad + 2 \RE (\chi_{_{R}} \psi_n, V [\chi_{_{R}}, \fw^{-1}] \fw
	\psi_n )_{_{L^2}} \\
	    \end{align*}
	    
	    Hence, we have 
	     \begin{align*}
	\abs{( \psi_n, \chi^2_{_{R}} V \psi_n )_{_{L^2}} -  & (
	\chi_{_{R}} \varphi_n , V^{^{2 \times 2}}_{_\text{FW}}
	\chi_{_{R}} \varphi_n )_{_{L^2}}}  \leq    | V^{1/2} [\fw^{-1}, \chi_{_{R}}]\fw \psi_n|^2_{_{L^{2}}} \\
	&\qquad  \qquad   + 2 |  \RE (\chi_{_{R}} \psi_n, V [\chi_{_{R}}, \fw^{-1}] \fw
	\psi_n )_{_{L^2}} | \\
	& \leq  | V^{1/2} [\fw^{-1}, \chi_{_{R}}]\fw \psi_n|^2_{_{L^{2}}} \\
	&\qquad  \qquad  +  2|V^{1/2} [ \chi_{_{R}} ,\fw^{-1}]
	\fw \psi_n |_{_{L^{2}}} |V^{1/2} \chi_{_{R}} \psi_n
	|_{_{L^{2}}}
	 \end{align*}

    Using lemma \ref{lem:estimate_weakspace} we have
    \begin{align*}
	&\abs{V^{1/2} \chi_{_{R}} \psi_n}_{_{L^{2}}} \leq C
	\abs{V}^{1/2}_{_{L^3_{w}}} \abs{ \psi_n}_{_{H^{1/2}}}  \\
	&|V^{1/2} [ \chi_{_{R}} ,\fw^{-1}] \fw \psi_n |_{_{L^{2}}} \leq
	C | V |^{1/2}_{_{L^3_{w}}} | [\chi_{_{R}}, \fw^{-1}] \fw
	\psi_n |_{_{H^{1/2}}}
    \end{align*}
    and it follows from lemma \ref{lem:stima_commutatori} (Appendix
    \ref{sec:appB}) that
    \begin{equation*}
	| [ \chi_{_{R}} ,\fw^{-1}] \fw \psi_n |_{_{H^{1/2}}} \leq
	\frac{C}{R} | \psi_n |_{_{H^{1/2}}} 
    \end{equation*}
    and hence
     \begin{equation}
	\label{eq:stima_potenziale}
	\abs{( \psi_n, \chi^2_{_{R}} V \psi_n )_{_{L^2}} - (
	\chi_{_{R}} \varphi_n , V^{^{2 \times 2}}_{_\text{FW}}
	\chi_{_{R}} \varphi_n )_{_{L^2}}} \leq  \frac{C}{R}  | V |_{_{L^3_{w}}}  | \psi_n |^2_{_{H^{1/2}}}  
	\to 0 \quad {\text{as}} \, R \to +\infty.
    \end{equation}

       Step 1 follows.
    
    \subsection*{Step 2:} $\abs{(\chi_{_{R}}  \varphi_n, V^{^{2 \times 2}}_{_\text{FW}}
  \chi_{_{R}}   \varphi_n )}_{_{L^2}}  \to 0$ as $n \to
    +\infty$.

     We have, by assumption, $\mathcal{I}(\phi_n) \to \lambda$,
    $\mathcal{G}(\phi_n) = |\varphi_n|^2_{_{L^2}} =1 $ and 
    \begin{equation*}
	\norm{\nabla_{S} \mathcal{I}(\phi_{n})} = \|
	d\mathcal{I}(\phi_n ) - \mu_{n} d \mathcal{G}(\phi_n ) \| \to
	0.
    \end{equation*}
where $\mu_n = \mu(\phi_n) $ and also, by \eqref{eq:stime_moltiplicatore} in particular, 
\begin{equation}
    \mu_n = \mathcal{I}(\phi_n) -  \frac{1}{2} (\nabla_{S}
    \mathcal{I}(\phi_n),\phi_n)_{_{H^{1}}} \to \lambda. 
\end{equation}

     Using \eqref{eq:differenziale_commutato}  we have

    \begin{align*}
	o_n(1) &= \|\nabla \mathcal{I}(\phi_n) - \mu_{n}
	\nabla\mathcal{G} (\phi_n)\| \| \phi_n \|_{_{H^1}} \geq
	\abs{(\nabla \mathcal{I}(\phi_n) - \mu_{n} \nabla\mathcal{G}
	(\phi_n), \chi_{_{R}}^{2}\phi_n )_{_{H^1}}} \\
	&\geq \abs{d\mathcal{I}(\phi_n) [\chi^2_{_{R}} \phi_n]} -
	\abs{\mu_{n} 2 \RE (\varphi_n , \chi^2_{_{R}}
	\varphi_n)_{_{L^2}}} \\
	&\geq 2 \mathcal{I}(\chi_{_{R}}\phi_{n})  -  2 \abs{\mu_{n}}  \abs{\chi_{_{R}} \varphi_{n}}_{_{L^{2}}}^{2} -
	2c^{2}  |\phi_{n}\nabla_{y}\chi_{_{R}}|^{2}_{_{L^{2}}} \\
	&  - 2 \abs{(V \psi_{n}, [\fw^{-1}, \chi_{_{R}}] \chi_{_{R}} \fw
	\psi_{n})_{_{L^{2}}} } 
	 - 2 \abs{([\chi_{_{R}}, \fw^{-1}] \fw \psi_{n}, V
	\fw^{-1} \chi_{_{R}} \fw \psi_{n})_{_{L^{2}}} }  \\
	&\geq 2 \mathcal{I}(\chi_{_{R}}\phi_{n})  -  2 \abs{\mu_{n}}  \abs{\chi_{_{R}} \varphi_{n}}_{_{L^{2}}}^{2}  - (I)-(II)-(III)
  \end{align*}
    
    Now, by Sobolev compact embedding, for any given $R >0$,
    \begin{equation*}
	| \chi_{_{R}} \varphi_n |_{_{L^2}}
	\to 0 \qquad \text{as} \quad n \to + \infty.
     \end{equation*}
     
     Moreover, 
      \begin{equation*}
       (I)  =  2c^{2}\abs{\phi_{n}\nabla \chi_{_{R}}}^{2}_{_{L^{2}}}  \leq  2 c^2 \sup_{y \in \R^3}  \abs{\nabla \chi_{_{R}}}^{2} \leq \frac{C}{R^2}\\
     \end{equation*}
     and  from lemma \ref{lem:stima_commutatori} (Appendix
    \ref{sec:appB}) we have

     \begin{align*}
    (II) &= 2\abs{(V \psi_{n}, [\fw^{-1}, \chi_{_{R}}]  \chi_{_{R}} \fw
	\psi_{n})_{_{L^{2}}}} \\
	&\leq  |V^{1/2} [ \chi_{_{R}} ,U_{_{FW}}^{-1}]
	 \chi_{_{R}} U_{_{FW}} \psi_n |_{_{L^{2}}} |V^{1/2}
	 \psi_n|_{_{L^{2}}} \\
	 & \leq \frac{C}{R}  | V |_{_{L^3_{w}}}  | \psi_n |^2_{_{H^{1/2}}}  \\
	    (III)&= 2 \abs{([\chi_{_{R}}, \fw^{-1}] \fw \psi_{n}, V 
	\fw^{-1} \chi_{_{R}} \fw \psi_{n})_{_{L^{2}}} } \\ 
	&\leq  |V^{1/2} [ \chi_{_{R}} ,U_{_{FW}}^{-1}]
	 U_{_{FW}} \psi_n |_{_{L^{2}}} |V^{1/2} U_{_{FW}}^{-1}
	 \chi_{_{R}} U_{_{FW}} \psi_n |_{_{L^{2}}} \\
	& \leq \frac{C}{R}  | V |_{_{L^3_{w}}}  | \psi_n |^2_{_{H^{1/2}}}  .
    \end{align*}

    Since by Lemma \ref{lem:coercive}-(i) we have
    \begin{equation*}
	\mathcal{I}(\chi_{_{R}} \phi_n) \geq \delta \|\chi_{_{R}}
	\phi_n \|^2_{_{H^1}}
    \end{equation*}
    we may conclude (recalling that $\mu_{n}$ is bounded) that
      \begin{equation*}
	\|\chi_{_{R}} \phi_n \|^2_{_{H^1}} 
	\leq \epsilon_n + \frac{C}{R}.
    \end{equation*} 
   and hence by {\bf (h2)} and lemma \ref{lem:betterminimizer} we get
    \begin{equation*}
	| ( \chi_{_{R}} \varphi_n , V^{^{2 \times 2}}_{_\text{FW}}
	\chi_{_{R}} \varphi_n )_{_{L^2}} | \leq a \|\chi_{_{R}} \phi_n
	\|^2_{_{H^1}} \leq \epsilon_n +
	\frac{C}{R}
    \end{equation*} 
   for some $\epsilon_{n} \to 0$ and $R$ arbitrarily large,  and step 2 follows. Then the  lemma follows from  step 1 and 2. 
\end{proof}

\begin{rem}
    \label{eq:positivita_quadratica}
    We recall that for all $v \in 
    C^{\infty}_{0}(\mathbb{R}^{4})$
    \begin{equation*}
	\int_{\mathbb{R}^{3}} |v(0,y)|^{2} dy = \int_{\mathbb{R}^{3}}
	dy \int_{+\infty}^{0} \partial_{x} |v|^{2} dx \leq 2 \| v
	\|_{_{L^2(\mathbb{R}^{4}_{+})}} \| \partial_{x} v
	\|_{_{L^2(\mathbb{R}^{4}_{+})}} 
    \end{equation*}
    and by density we get for all $\phi \in	H^1$
    \begin{equation*}
	m c^2 \int_{\mathbb{R}^{3}}
	|\phi_{_{\text{tr}}}|^{2} \, dy \leq
	\iint_{\R^4_{+}} ( |\partial_{x} \phi |^2+ m^2
	c^4 |\phi |^2) \, dx dy 
	   \end{equation*}
   Hence the quadratic form (kinetic energy)
   \begin{equation*}
       \mathcal{T}(\phi) =  \iint_{\R^4_{+}} ( |\partial_{x}
       \phi |^2 + m^2 c^4 |\phi |^2 ) \, dx dy
       - mc^2 \| \phi_{_{tr}} \|^2_{_{L^2}} 
   \end{equation*}
   is positive definite.
\end{rem}
 
Now we may conclude the existence of a minimizer for
$\mathcal{P}_{_{1}}$.  We have the following proposition:

\begin{prop}
    \label{prop:ground_state}
    Let $\phi_n $ be a minimizing Palais Smale sequence at level
    $\lambda_1 > 0$ for $\mathcal{I}$ with $ |(\phi_n)_{tr}|_{_{L^2}}
    =1 $ (as in Lemma \ref{lem:potential_term}).
    
    Then $\phi_n \rightharpoonup \phi \not \equiv 0$ in $H^1$ and $\hat{\phi} = \abs{\phi}_{_{L^{2}}}^{-1}\phi$ is
    a minimizer for $\mathcal{I}$ on $S$, that is
    \begin{equation*}
	\mathcal{I} (\hat{\phi}) = \lambda_1, \qquad
	\abs{\hat{\phi}}_{_{L^{2}}} = 1.
    \end{equation*}
    
    Moreover $\hat{\phi}$ (and hence also $\phi$) is a weak solution
    of the Neumann problem $(\mathcal{E}_{_{1}})$.
\end{prop}

\begin{proof}
    Since $ \mathcal{I}$ is coercive, $\phi_n $ is bounded (and weakly
    convergent) in $H^{1}$, $\varphi_n =
    (\phi_{n})_{tr}$ is bounded (and weakly convergent)
    in $H^{1/2}$.
 
    If by contradiction $\varphi_n \rightharpoonup \varphi \equiv 0$,
    then by lemma \ref{lem:potential_term} we have
    \begin{equation*}
	( \varphi_n, V^{^{2 \times 2}}_{_\text{FW}} \varphi_n
	)_{_{L^2}} \to 0.
    \end{equation*} 
    Now, by Remark \ref{eq:positivita_quadratica} we get
    \begin{equation*}
	\mathcal{I}(\phi_n) - mc^2 | \varphi_n |^2_{_{L^2}} \geq
	( \varphi_n, V^{^{2 \times 2}}_{_\text{FW}} \varphi_n
	)_{_{L^2}} \to 0 .
   \end{equation*}
   On the other hand, by Lemma \ref{lem:coercive}-(ii)
   \begin{equation*}
       \mathcal{I}(\phi_n) - mc^2 | \varphi_n |^2_{_{L^2}} =
       \mathcal{I}(\phi_n) - mc^2 \to \lambda_1 - mc^2 < 0
   \end{equation*}
   a contradiction, that is  $ \varphi_n \rightharpoonup \varphi \not
   \equiv 0$.

   It follows from \eqref{eq:stime_moltiplicatore} that
   \begin{equation*}
       \mu_{n} = \mathcal{I}(\phi_{n}) - 
     \frac{1}{2}  (\nabla_{S}\mathcal{I}(\phi_{n}), \phi_{n})_{_{H^1}} \to \lambda_{1}
   \end{equation*}
   and hence, by weak convergence, we have
   \begin{equation*}
       d\mathcal{I}(\phi_n)[h] - \mu_{n} d\mathcal{G}(\phi_{n})[h] \to
       d\mathcal{I}(\phi) [h] - \lambda_1 d\mathcal{G}(\phi) [h] = 0
       \quad \forall h \in H^{1}
          \end{equation*}
   hence in particular
   \begin{equation*}
       0= d\mathcal{I}(\phi) [\phi] - \lambda_1 d\mathcal{G}(\phi)
       [\phi] = 2 \mathcal{I}(\phi) - 2 \lambda_1 \mathcal{G}(\phi)
   \end{equation*} 
   and we may conclude that $\hat{\phi} =
   \mathcal{G}(\phi)^{-1/2}\phi$ is a minimizer for $\mathcal{I}$ on
   $S$, namely
   \begin{align*}
       &\lambda_1 = \frac{\mathcal{I} (\phi)}{\mathcal{G}(\phi)} =
       \mathcal{I}(\mathcal{G}(\phi)^{-1/2}\phi) = \mathcal{I}(\hat{\phi})\\
       &\mathcal{G}(\hat{\phi}) =
       \mathcal{G}(\mathcal{G}(\phi)^{-1/2}\phi) = 1
   \end{align*}
\end{proof}
 
Now, we look for the existence of higher eigenvalues and corresponding
eigenfunctions.  We proceed by induction.

Let $\lambda_{1}$ be defined by \eqref{eq:p1} and $\phi_{1}$ be the
corresponding minimizer given by Proposition \ref{prop:ground_state}.

Assume we have defined, for $j=1, \dots, k-1$, $\lambda_1 \leq \dots \lambda_j < m c^2 $ and $\phi_j \in H^{1}$,
$\varphi_j = (\phi_{j})_{tr} \in
H^{1/2}$ such that
\begin{equation*}
    (\varphi_{i}, \varphi_{j})_{L^{2}} = \delta_{ij}, \qquad i,j =
    1,\ldots,k-1,
\end{equation*}
and
\begin{equation}
    \label{eq:pj}
    \tag{$\mathcal{P}_{_{j}}$} \lambda_{j} = \mathcal{I}(\phi_{j}) =
    \inf_{_{ \phi \in X_{j}} } {\mathcal{I} (\phi)} \qquad j = 1,
    \ldots, k-1
\end{equation}
where,
\begin{equation*}
    X_{j} = \settc{\phi \in H^{1}}{\mathcal{G}(\phi) =
    \abs{\phi_{tr}}_{L^{2}}^{2} = 1, \ (\phi_{tr}, \varphi_i)_{_{L^2}}
    = 0 \quad \text{for} \quad i=1, \dots, j-1}.
\end{equation*}

We define 
\begin{equation}
    \label{eq:pk}
    \tag{$\mathcal{P}_{_{k}}$} \lambda_{k} = \inf_{_{ \phi \in X_{k}}
    } \, \mathcal{I} (\phi)
\end{equation}
 
\begin{rem}
    Setting $\mathcal{G}_j(\phi) = (\varphi_j, \phi_{tr})_{_{L^2}}$,
    for $j \geq 1$, we have that the linear functionals
    $\mathcal{G}_j$ are bounded on $H^{1}$ and for any $\phi , h \in H^1$
    \begin{equation*}
	d\mathcal{G}_j (\phi)[h ] = (\nabla
	\mathcal{G}_{j} (\phi), h)_{_{H^{1}}} = ( \varphi_j, h_{_{tr}}
	)_{_{L^2}}  = \mathcal{G}_{j}(h) \qquad j= 1, \ldots k-1.
    \end{equation*}
    Then $X_{k} = \settc{\phi \in X_{1}}{\mathcal{G}(\phi) = 1, \
    \mathcal{G}_{j}(\phi) = 0, \ j = 1, \ldots, k-1}$,
    \begin{equation*}
	T_{\phi} X_{k} = \settc{h \in H^{1}}{(\nabla
	\mathcal{G}(\phi), h)_{H^{1}} = 0, \ \mathcal{G}_{j}(h) = 0, \
	j=1,\ldots,k-1}
    \end{equation*}
    and the constrained gradient (i.e. the projection of the gradient
    of $\mathcal{I}$ on the tangent space $T_{\phi} X_{k}$) is given
    by
    \begin{equation*}
	\nabla_{X_{k}} \mathcal{I}(\phi) = \nabla \mathcal{I}(\phi) -
	\mu_{0}(\phi) \nabla \mathcal{G}(\phi) - \sum_{j=1}^{k-1}
	\mu_{j}(\phi) \nabla \mathcal{G}_{j}(\phi).
    \end{equation*}
    From
    \begin{align*}
	(\nabla_{X_{k}} \mathcal{I}(\phi), \phi)_{_{H^{1}}} &= (\nabla
	\mathcal{I}(\phi), \phi)_{_{H^{1}}} - \mu_{0}(\phi) (\nabla
	\mathcal{G}(\phi), \phi)_{_{H^{1}}} - \sum_{j=1}^{k-1}
	\mu_{j}(\phi) (\nabla \mathcal{G}_{j}(\phi), \phi)_{_{H^{1}}} \\
	&= 2\mathcal{I}(\phi) - 2\mu_{0} (\phi)\mathcal{G}(\phi) -
	\sum_{j=1}^{k-1} \mu_{j}(\phi) \mathcal{G}_{j}(\phi) =
	2\mathcal{I}(\phi) - 2\mu_{0}(\phi) 
    \end{align*}
     for $\phi \in X_{k}$,  we deduce that 
    \begin{equation}
	\label{eq:mu0}
	\mu_{0}(\phi) = \mathcal{I}(\phi) - \frac{1}{2}
	(\nabla_{X_{k}} \mathcal{I}(\phi), \phi)_{_{H^{1}}}
    \end{equation}
     while for $\phi \in X_{k}$ and $\varphi_{i} = (\phi_i)_{tr} $,  for $i = 1, \dots,  k-1$, from
 
     \begin{align*}
	(\nabla_{X_{k}} \mathcal{I}(\phi), \phi_{i})_{_{H^{1}}} = &
	(\nabla \mathcal{I}(\phi), \phi_{i})_{_{H^{1}}} - \mu_{0}(\phi)
	(\nabla \mathcal{G}(\phi), \phi_{i})_{_{H^{1}}} -
	\sum_{j=1}^{k-1} \mu_{j}(\phi) (\nabla \mathcal{G}_{j}(\phi),
	\phi_{i})_{_{H^{1}}} \\
	= & d\mathcal{I}(\phi)[\phi_{i} ] - \mu_{0} (\phi) 2  \RE (\phi_{tr},
	\varphi_{i})_{_{L^{2}}} -  \sum_{j=1}^{k-1} \mu_{j}(\phi) (\varphi_j, \varphi_i)_{_{L^2}}\\
	=&  d\mathcal{I}(\phi)[\phi_{i} ] -  \mu_i(\phi)
  \end{align*}
    
   we have that
   
    \begin{equation}
	\label{eq:muJ}
	\mu_{i}(\phi)  = d\mathcal{I}(\phi)[\phi_{i} ] 	- (\nabla_{X_{k}} \mathcal{I}(\phi), \phi_{i})_{_{H^{1}}}
    \end{equation}

    We say that $ \phi_n \in X_k$ is a (constrained) Palais Smale
    sequence for $\mathcal{I}$ on $X_{k}$ at level $\lambda_k$ if
    $\phi_{n} \in X_{k}$,
    \begin{equation*}
	\mathcal{I}(\phi_n) \to \lambda_k \quad \text{and} \quad
	\|\nabla_{_{X_{k}}} \mathcal{I} (\phi_n) \| \to 0.
    \end{equation*}
\end{rem}

The proof of existence of a minimizer for \eqref{eq:pk}
proceeds as the proof of the existence of the
ground state $\phi_1$.  The key points are the following two lemmas.
\begin{lem}
    $\lambda_1\leq \lambda_{k} < m c^2$.
\end{lem}
 
\begin{proof}
    Let us consider any $k$-dimensional linear subspace $G_{k} \subset
    C^{\infty}_{0}(\R^{3}; \C^2 )$. 

    For $\varphi \in G_{k} \cap S$ and $\eta > 0$ we let
    $\varphi_{\eta}(y) = \eta^{3/2} \varphi(\eta y) \in S$ and
    \begin{equation*}
	F^{\eta}_{k} = \settc{ \phi_{_{\eta}} \in
	H^{1}} {\phi_{_{\eta}} (x,y) = \esp^{-
	mc^2 x} \varphi_{_{\eta}}(y) , \quad \varphi \in G_{k} \cap S}.
    \end{equation*}
    
    Then, for any $\phi_{_{\eta}} \in F^{\eta}_{k}$
    \begin{align*}
	\mathcal{I} (\phi_{_{\eta}}) - mc^2 & =\frac{1}{2 m}
	\int_{\R^3} |\nabla \varphi_{_{\eta}}|^2+
	\int_{\R^3} ( \varphi_{_{\eta}}, V^{^{2 \times 2}}_{_\text{FW}}
	\varphi_{_{\eta}})_{_{\C^2}} \\
	&=  \frac{\eta ^2}{2 m} \int_{\R^3} |\nabla \varphi
	|^2 + \int_{\R^3} ( \varphi_{_{\eta}}, V^{^{2 \times
	2}}_{_\text{FW}} \varphi_{_{\eta}} )_{_{\C^2}}
    \end{align*}
    Arguing as in Lemma \ref{lem:coercive}-(ii) and by compactness of
    the set $G_{k} \cap S$, there exists $\bar{\eta} >0$ such that for
    any $\phi_{_{\bar{\eta}}} \in F^{\bar{\eta}}_{k}$, we have
    \begin{equation*}
	{\bar{\eta} ^2} \frac{1}{2 m} \int_{\R^3} |\nabla \varphi
	|^2+ \int_{\R^3} ( \varphi_{_{\bar{\eta} }}, V^{^{2
	\times 2}}_{_\text{FW}} \varphi_{_{\bar{\eta} }} )_{_{\C^2}} <
	0
    \end{equation*}
    
    Since  $X_{k} \cap F^{\bar{\eta}}_{k} \not= \emptyset$, we have 
    $\lambda_{k} \leq \sup_{ F^{\bar{\eta}}_{k}} \mathcal{I} (
    \phi_{_{\bar{\eta}}}) < m c^2$ .
\end{proof}
 
\begin{lem}
    \label{lem:potential_term_k}
    Let $\zeta_n \in X_{k}$ be a (constrained) Palais Smale sequence at
    level $\lambda_k$ for $\mathcal{I}$ on $X_{k}$.  

    Then, as $n \to + \infty$
       \begin{equation*}
        \mu_{0}(\zeta_{n})  \to \lambda_{k}  \qquad \qquad 
        \mu_{j}(\zeta_{n})  \to 0  \qquad (j= 1, \dots, k-1)
    \end{equation*}

     Moreover, if $\xi_{n}  = 
    (\zeta_n)_{tr} \rightharpoonup 0$ in $H^{1/2}$ then
    \begin{equation*}
	( \xi_n, V^{^{2 \times 2}}_{_\text{FW}} \xi_n
	)_{_{L^2}} \to 0.
    \end{equation*} 
\end{lem}

\begin{proof}
    We have that $\zeta_{n} \in X_{k}$ is such that 
    \begin{equation*}
	\mathcal{I}(\zeta_n) \to \lambda_k \quad \text{and} \quad
	\|\nabla_{_{X_{k}}} \mathcal{I} (\zeta_n) \| \to 0.
    \end{equation*}

    Then $\zeta_{n}$ is bounded and from \eqref{eq:mu0} and
    \eqref{eq:muJ} we have, as $n \to + \infty$
    \begin{align*}
	\mu_{0}(\zeta_{n}) &= \mathcal{I}(\zeta_{n}) - \frac{1}{2}
	(\nabla_{X_{k}} \mathcal{I}(\zeta_{n}), \zeta_{n})_{_{H^{1}}}
	\to \lambda_k \\
	\mu_{j}(\zeta_{n}) &= d\mathcal{I}(\zeta_{n})[\phi_j] -
	(\nabla_{X_{k}} \mathcal{I}(\zeta_{n}), \phi_{j})_{_{H^{1}}} =
	d\mathcal{I}(\phi_j)[\zeta_{n}] - (\nabla_{X_{k}}
	\mathcal{I}(\zeta_{n}), \phi_{j})_{_{H^{1}}} \\
	& = 2 \lambda_j \RE (\xi_n, \varphi_j)_{_{L^2}} -
	(\nabla_{X_{k}} \mathcal{I}(\zeta_{n}), \phi_{j})_{_{H^{1}}} =
	- (\nabla_{X_{k}} \mathcal{I}(\zeta_{n}), \phi_{j})_{_{H^{1}}}
	\to 0
    \end{align*}
    for $j=\, \dots, k-1$.

    We then proceed as in the proof of lemma \ref{lem:potential_term}.
    
    In particular the first step in the proof of that lemma holds also
    here, that is 
    \begin{equation*}
	\abs{(\xi_n, V^{^{2 \times 2}}_{_\text{FW}} \xi_n )_{_{L^2}} -
	(\chi_{_{R}}\xi_n, V^{^{2 \times 2}}_{_\text{FW}} \chi_{_{R}}
	\xi_n )_{_{L^2}}} \to 0 \text{ as } R \to +\infty.
    \end{equation*}
    
    We now use again \eqref{eq:differenziale_commutato} and similar to
    step 2 of lemma \ref{lem:potential_term} and the fact that
    $\zeta_{n}$ is a constrained Palais Smale sequence, we have
        \begin{align*}
	&o_n(1) = \|\nabla_{_{X_{k}}} \mathcal{I} (\zeta_n) \|_{_{H^1}}
	\|\zeta_n \|_{_{H^1}} \geq \abs{(\nabla_{_{X_{k}}} \mathcal{I}
	(\zeta_n) , \chi_{_{R}}^{2} \zeta_n)}_{_{H^1}} \\
	&\geq \abs{d\mathcal{I}(\zeta_n) [\chi^2_{_{R}} \zeta_n]} -
	\abs{\mu_{0}(\zeta_{n}) (\nabla \mathcal{G}(\zeta_{n}) ,
	\chi^2_{_{R}} \zeta_n)_{H^{1}}} - \abs{\sum_{ j=1}^{ k -1}
	\mu_j (\zeta_n) (\nabla \mathcal{G}_j(\zeta_{n}), \chi^2_{_{R}} \zeta_n
	)_{H^{1}}} \\
		& \geq 2 \, \mathcal{I}(\chi_{_{R}} \zeta_n) -
	2\abs{\mu_{0}(\zeta_{n})} \abs{\chi_{_{R}} \xi_n}^{2}_{L^2} - 
	\sum_{ j=1}^{ k -1} \abs{\mu_j (\zeta_n) }\abs{\chi_{_{R}}
	\xi_n}_{L^{2}}  - (I) - (II)  - (III) 
    \end{align*} 
    where as in the proof of lemma \ref{lem:potential_term} (see also lemma \ref{lem:stima_commutatori} , Appendix
    \ref{sec:appB}) we have 
    \begin{align*}
    (I)  &=  2c^{2}\abs{\zeta_{n}\nabla \chi_{_{R}}}^{2}_{_{L^{2}}}  \leq  2 c^2 \sup_{y \in \R^3}  \abs{\nabla \chi_{_{R}}}^{2} \leq \frac{C}{R^2}\\
    (II) &= 2\abs{(V \psi_{n}, [\fw^{-1}, \chi_{_{R}}]  \chi_{_{R}} \fw
	\psi_{n})_{_{L^{2}}}} \\
	&\leq  |V^{1/2} [ \chi_{_{R}} ,U_{_{FW}}^{-1}]
	 \chi_{_{R}} U_{_{FW}} \psi_n |_{_{L^{2}}} |V^{1/2}
	 \psi_n|_{_{L^{2}}} \\
	 & \leq \frac{C}{R}  | V |_{_{L^3_{w}}}  | \psi_n |^2_{_{H^{1/2}}}  \\
	    (III)&= 2 \abs{([\chi_{_{R}}, \fw^{-1}] \fw \psi_{n}, V 
	\fw^{-1} \chi_{_{R}} \fw \psi_{n})_{_{L^{2}}} } \\ 
	&\leq  |V^{1/2} [ \chi_{_{R}} ,U_{_{FW}}^{-1}]
	 U_{_{FW}} \psi_n |_{_{L^{2}}} |V^{1/2} U_{_{FW}}^{-1}
	 \chi_{_{R}} U_{_{FW}} \psi_n |_{_{L^{2}}} \\
	& \leq \frac{C}{R}  | V |_{_{L^3_{w}}}  | \psi_n |^2_{_{H^{1/2}}}  
    \end{align*} 
              (here $\psi_{n} = \fw^{-1}Q\xi_{n}$)
              
  Then  by Sobolev compact  embedding,  for any
    given $R >0$,
    \begin{equation*}
	| \chi_{_{R}} \xi_n |_{_{L^2}} \to 0 \quad \text{as} \quad n
	\to + \infty.  
    \end{equation*}
      Moreover, $ |\mu_j (\zeta_n)| \leq C$ for $j=0, \dots, k-1$.  More precisely,

    Now, since $\mathcal{I}$ is coercive, exactly as in lemma \ref{lem:potential_term} we may conclude
    \begin{equation*}
	\|\chi_{_{R}} \zeta_n \|^2_{_{H^1}}
	\leq \epsilon_{n} + \frac{C}{R}
    \end{equation*} 
and by {\bf (h2)} and lemma \ref{lem:betterminimizer},
    \begin{equation*}
	|( \Vfw \chi_{_{R}} \xi_n , \chi_{_{R}} \xi_n
	)_{_{L^2}} | \leq a \|\chi_{_{R}} \zeta_n \|^2_{_{H^1}} \leq \epsilon_{n} + \frac{C}{R}
    \end{equation*} 
 for $\epsilon_n \to 0 $ as $n \to + \infty$,  $R$ arbitrary large,     and the lemma follows.
\end{proof}
 
We are now ready to prove the following proposition for  the
existence of a minimizer for \eqref{eq:pk}.
 
\begin{prop}
\label{prop:excited_states}
    Let $\zeta_n \in X_k$ be a minimizing Palais Smale sequence for
    \eqref{eq:pk}.
    
    Then $\zeta_n \rightharpoonup \phi_k $ in $H^1$ and $
    |(\phi_k)_{tr}|_{_{L^2}}^{-1} \phi_k \in X_k$ is a minimizer for
    problem \eqref{eq:pk}, and a weak solution of the Neumann problem
    $(\mathcal{E})_{_{k}}$.
\end{prop}

\begin{proof}
    We proceed as in the proof of lemma \ref{prop:ground_state} to
    conclude that $\zeta_n \rightharpoonup \phi_k \not \equiv 0 $.
    
    We clearly have that $\mathcal{G}_{j}(\phi_{k}) = 0$ for $j = 1,
    \ldots, k-1$.  We do not know if $\abs{\varphi_{k}}_{L^{2}} = 1$
    (where $\varphi_{k} = (\phi_{k})_{tr}$).
    
   By  lemma  \ref{lem:potential_term_k} we have 
that  
    \begin{equation*}
        \mu_{0}(\zeta_{n}) \to \lambda_{k}  \qquad 
        \mu_{j}(\zeta_{n}) \to 0  \qquad (j= 1, \dots, k-1)
    \end{equation*}
 
 then by weak convergence we then have that for all $h \in H^{1}$, as $n \to + \infty$
    \begin{align*}
	( \nabla_{X_{k}} \mathcal{I}(\zeta_n), h )_{_{H^1}} &=  d\mathcal{I}(\zeta_n) [h]
	 - 2\mu_{0}(\zeta_n) \RE (\xi_n, h_{tr})_{_{L^{2}}} -
	\sum_{ j=1}^{ k -1} \mu_{j}(\zeta_n) (\varphi_{j},
	h_{tr})_{_{L^{2}}} \\
	&\to d\mathcal{I}(\phi_k) [h] - 2 \lambda_k  \RE (\varphi_k,
	h_{tr})_{_{L^2}}  =0.
    \end{align*}
    We deduce, taking $h = \phi_{k}$
    \begin{equation*}
	0= d\mathcal{I}(\phi_k) [\phi_k] - \lambda_k 2
	|\varphi_k|^2_{_{L^2}} = 2 \mathcal{I}(\phi_k) - 2 \lambda_k
	|\varphi_k|^2_{_{L^2}}
    \end{equation*}
    and we conclude that $\abs{\varphi_{k}}_{_{L^{2}}}^{-1}\phi_k \in
    X_{k}$ is a minimizer for \eqref{eq:pk}.
  \end{proof}

To conclude the proof of Theorem \ref{thm:eigenvalues} we prove 
that $\{ \lambda_{k} \}_{_{k \geq 1}} \in
\sigma_{disc}(\mathcal{B}_{_\text{FW}}) $ namely that $\lambda_k$ has
finite multiplicity.  

Indeed suppose that there exists an eigenvalue $\lambda_{k}$ with
infinite multiplicity.  Then there exist a corresponding sequence $\{
\varphi^{(k)}_n\}_{n \in \mathbb{N}} \subset H^{1/2}$ of
eigenfunctions corresponding to the same eigenvalue $\lambda_k$.  We
will assume that $\abs{\varphi_{n}^{(k)}}_{_{L^{2}}} = 1$ for all $n \in
\mathbb{N}$.  Letting
\begin{equation*}
    \phi^{(k)}_{n} =\mathcal{F}_{_{y}}^{-1} \left[ e^{-x\sqrt{m^{2}
    c^4+ c^2 |p|^{2}}} \mathcal{F} [\varphi^{(k)}_n ]
    \right]\in X_{k},
\end{equation*}

by lemma \ref{lem:betterminimizer} we have $\nabla_{X_{k}} \mathcal{I}(\phi^{(k)}_{n}) =0 $ and $\mathcal{I}
( \phi^{(k)}_{n}) = \lambda_k$.  We deduce from this that
$\varphi_{n}^{k}$ is a bounded sequence in $H^{1/2}$, since by orthogonality  $ \varphi^{(k)}_n
\rightharpoonup 0 $ in $L^{2}$, we have $ \varphi^{(k)}_n
\rightharpoonup 0 $   in $H^{1/2}$, therefore by
lemma \ref{lem:potential_term} we get
\begin{equation*}
    ( \varphi^{(k)}_n , V^{^{2 \times 2}}_{_\text{FW}}
    \varphi^{(k)}_n )_{L^2} \to 0
    \qquad \text{as } n \to +\infty
\end{equation*}
and from this we get  a contradiction, namely $\lambda_{k} = \mathcal{I} (\phi^{(k)}_{n})
\geq mc^2$.  
 
Finally since eigenvalues can accumulate only on the essential spectrum, we
may conclude  that
\begin{equation*}
    0 < \lambda_1 \leq ....\leq \lambda_{k-1} \leq \lambda_{k} \, \to
    \, \inf \{ \sigma_{\text{ess}} (\mathcal{B}_{_\text{FW}}) \} =
    mc^2 \quad \text{for} \, \, k \to + \infty .
\end{equation*}
    
\appendix

\section{Multiplication and convolutions on Lorentz spaces}
\label{sec:appendixA}

We refer to \cite{Ziemer} for the definition of the Lorentz space
$L(p,q)$, for $1 \leq p,q \leq \infty $ and the corresponding norm $\|
f \|_{(p,q)}$.  Let us recall here the following facts and
inequalities (see e.g. \cite {Ziemer}, \cite{Tartar_1998} and
references therein for more details)
    
\begin{itemize}
    
    \item for $p=q$ the Lorentz spaces $L(p,p)$ coincide with the
    usual $L^p$-space.
    
    \item for $p >1$ and $q= \infty$ the Lorentz spaces $L(p, \infty)$
    correspond to the weak $L^p$-space $L^p_{w} (\R^N)$ (also known
    as Marcinkiewicz spaces $M_{p} (\R^N)$).
    
    \item {\it multiplication:} Let $f \in L(p_1,q_1)$ and $g \in
    L(p_2,q_2)$ then $fg \in L(r,s)$ with $\frac{1}{p_1} +
    \frac{1}{p_2} = \frac{1}{r}$ and $\frac{1}{q_1} + \frac{1}{q_2} =
    \frac{1}{s}$ for $1 < p_1, p_2 < \infty, 1 \leq q_1, q_2 \leq
    \infty$ .
    
    \item {\it convolution:} Let $f \in L(p_1,q_1)$ and $g \in
    L^{(p_2,q_2)} (\R^N)$ then $f \ast g \in L(r,s)$ with
    $\frac{1}{p_1} + \frac{1}{p_2} = \frac{1}{r} +1 $ and
    $\frac{1}{q_1} + \frac{1}{q_2} \geq \frac{1}{s} $ for $1 < p_1,
    p_2 < \infty, 1 \leq q_1, q_2 \leq \infty$ and $\frac{1}{p_1} +
    \frac{1}{p_2} >1$.  Moreover
    \begin{equation*}
	\|f \ast g \|_{(r,s)} \leq 3r \|f \|_{(p_1, q_1)} \|g
	\|_{(p_2, q_2)} .
    \end{equation*}

\end{itemize}
    
Finally let us point out the following generalization of the weak
Young inequality.

\begin{prop}[\protect{see \cite[thm. 2.10]{Kovalenko_1981}}]
    \label{prop: Kovalenko}
    Let $f \in L^{q}_{w}(\R^N)$, $g \in L^{q'}_{w}(\R^N)$ and $h \in
    L^{p}(\R^N)$ with $\frac{1}{q} + \frac{1}{q'} = 1$ and $1 < p < q
    $.  Then
    \begin{equation}
	\| f (g \ast h) \|_{p} \leq C \|f\|_{q,w} \|g\|_{q',w} \|h
	\|_{p}.
    \end{equation}
\end{prop}

\section{Estimates on commutators}
\label{sec:appB}
        
\begin{lem}
    \label{lem:stima_commutatori}      
    Let $\chi \in C_0^{\infty}(\R^3)$ and for $R >0$ let define
    $\chi_{_{R}}(y) = \chi ( R^{-1} y)$.  Then the operator $ [
    \chi_{_{R}} ,\fw^{-1}] \fw : H^{1/2} (\R^3; \C^4) \to H^{1/2} (\R^3; \C^4)$
    satisfies
    \begin{equation*}
	\| [ \chi_{_{R}} ,\fw^{-1}] \fw \| = O(R^{-1})  \qquad \text{as} \quad R \to + \infty.
    \end{equation*}
\end{lem}
      
\begin{proof}
    We have 
    \begin{equation*}
	\norm{[ \chi_{_{R}} ,\fw^{-1}] \fw \psi}^2_{_{H^{1/2}}} =
	\int_{\R^3} dp \, (1 + |p|) \left| \mathcal{F}( [ \chi_{_{R}}
	,\fw^{-1}] \fw \psi ) \right|^2
    \end{equation*}
    where
    \begin{align*}
	\mathcal{F}( [ \chi_{_{R}} ,\fw^{-1}] \fw \psi ) (p) &=
	\mathcal{F}( \chi_{_{R}} \psi ) (p) - \mathcal{F}( \fw^{-1}
	\chi_{_{R}} \fw \psi ) (p) \\
	&= (\hat{\chi}_{_{R}} * \hat{\psi}) (p) - U^{-1}(p)
	\mathcal{F} (\chi_{_{R}} \fw \psi ) (p) \\
	&= \int_{\mathbb{R}^{3}} R^{3}\hat{\chi}(Rq) \hat{\psi}(p-q)
	\, dq - U^{-1}(p) (\hat{\chi}_{_{R}} * (U \hat{\psi} )) (p) \\
	&=U^{-1}( p) \int_{\R^3} dq \, R^3 \, \hat{\chi} (Rq) \left(U(
	p) - U(p-q) \right) \hat{ \psi}(p - q ) \, dq \\
    \end{align*}
    where $\hat{\chi}$ is the Fourier transform of $\chi \in
    C_0^{\infty}(\R^3)$ (and hence $\hat{\chi}_{_{R}}(p) =
    R^{3}\hat{\chi}(Rp)$).

    Now, let define $\Rfw = \mathcal{F}^{-1} U(R^{-1} p)\mathcal{F} $
    and $ \hat{ \psi}_{_R}(q )= R^{- 3/2} \hat{ \psi} ( R^{- 1} q) $
    where $\hat{\psi}$ is the Fourier transform of $\psi \in
    H^{1/2}(\R^3)$.  By rescaling variables we get
    \begin{multline*}
	\norm{[ \chi_{_{R}} ,\fw^{-1}] \fw \psi}^2_{_{H^{1/2}}} \\
	= \int_{\R^3} dp \, (1 + R^{-1} |p|) \labs{U^{-1}( R^{- 1}p)
	\int_{\R^3} dq \, \hat{\chi}(q)K_{R}(p, q) \hat{
	\psi}_{_R}(p-q)}^2 \\
	= \int_{\R^3} dp \, (1 + R^{-1} |p|) \labs{
	\int_{\R^3} dq \, \hat{\chi}(q)K_{R}(p, q) \hat{
	\psi}_{_R}(p-q)}^2 \\
    \end{multline*}
 where
    \begin{align*}
	K_{R}(p,q) &= U( R^{- 1} p) - U( R^{- 1}(p-q)) \\
	& = \left( a_{+} (R^{- 1} p) - a_{+} ( R^{- 1}(p-q)) \right)
	\, \I_{4} \\
	&\qquad + \sum_{i=1}^{3} \left(a_{-} (R^{- 1} p )  \frac{p_i}{|p|} - a_{-} ( R^{- 1}(p-q))
	\frac{(p_i-q_i)}{|p-q|}  
	\right) \beta \mathbf{\alpha}_i \\
	&=: K_{+,R}(p,q) \I_{4} + \sum_{i=1}^{3}K^{i}_{-,R}(p,q) \beta
	\mathbf{\alpha}_i
    \end{align*}
    and $a_{\pm}(R^{- 1}p) = \sqrt{\frac{1}{2}( 1 \pm mc^2 /
    \lambda(R^{- 1}p))}$.  
    
    Letting $\lambda_{_{R}}( p) = \lambda(R^{-
    1} p) = \sqrt{c^2 R^{-2}|p|^2 + m^2 c^4}$ we have
    \begin{align*}
	| K_{+,R} &(p,q) | = |a_{+} (R^{- 1} p) - a_{+} ( R^{-
	1}(p-q)) | = \frac{|a^2_{+} (R^{- 1} p) - a^2_{+} ( R^{-
	1}(p-q)) |}{ a_{+} (R^{- 1} p) + a_{+} ( R^{- 1}(p-q))} \\
	& \leq \frac{m c^2}{2 \sqrt{2}} \left|
	\frac{1}{\lambda_{_{R}}(p)} - \frac{1}{\lambda_{_{R}}(p-q)}
	\right| \leq \frac{m c^2}{2 \sqrt{2}} \frac{ |\lambda_{_{R}}^2
	(p-q) - \lambda_{_{R}}^2(p) |}
	{\lambda_{_{R}}(p)\lambda_{_{R}}(p-q)( \lambda_{_{R}}(p) +
	\lambda_{_{R}}(p-q) )} \\
	& \leq \frac{m c^2}{2 \sqrt{2}} \frac{c^2 R^{-2}| |p-q|^2 -
	|p|^2|}{\lambda_{_{R}}( p)\lambda_{_{R}}(p-q) cR^{-1}(|p| +
	|p-q|) }  \leq
	\frac{\sqrt{2} |q|}{ 4 m c R}
    \end{align*}    
    and analogously, 
  
    \begin{align*}
	| K^{i}_{-,R}&(p,q) |  \leq \left| a_{-} (R^{- 1} (p-q))
	\frac{ (p_i-q_i)}{|p-q|} - a_{-} ( R^{- 1}p) \frac{ p_i}{|p|}
	\right| \\
	& \leq \, |p| \left | \frac{a_{-} (R^{- 1} (p-q))}{|p-q|} -
	\frac{a_{-} ( R^{- 1}p) }{|p|} \right| + |q| \frac{ a_{-}
	(R^{- 1} (p-q)) }{|p-q|} \\
	& \leq \frac{1}{\sqrt{2}} \left| \frac{|p|}{|p-q|} \left(
	\frac{\lambda_{_{R}}(p-q) - mc^2}{\lambda_{_{R}} (p-q) }
	\right)^{\frac{1}{2}} - \left( \frac{\lambda_{_{R}}( p) -
	mc^2}{\lambda_{_{R}}( p )} \right)^{\frac{1}{2}} \right| \\
	& \qquad + \frac{1}{\sqrt{2}} \frac{|q|}{|p-q|} \left(
	\frac{\lambda_{_{R}} (p-q) - mc^2}{\lambda_{_{R}}(p-q)}
	\right)^{\frac{1}{2}} \\
	& \leq \frac{|p|}{\sqrt{2}|p-q|}  \left|\left(
	\frac{c^2 R^{-2} |p-q|^2}{\lambda_{_{R}} (p-q) (\lambda_{_{R}}
	(p-q) + mc^2)} \right)^{\frac{1}{2}} - \left( \frac{c^2 R^{-2}
	|p|^2}{\lambda_{_{R}}( p )(\lambda_{_{R}} (p) + mc^2)}
	\right)^{\frac{1}{2}} \right| \\
	& \qquad + \frac{1}{\sqrt{2}} \frac{|q|}{|p-q|} \left(
	\frac{c^2 R^{-2} |p-q|^2}{\lambda_{_{R}}(p-q) (\lambda_{_{R}}
	(p-q) + mc^2)} \right)^{\frac{1}{2}} \\
		& \leq \frac{1}{\sqrt{2}} \frac{c |p| }{ R} \frac{ \left|
	\left( \lambda_{_{R}}( p )(\lambda_{_{R}} (p) + mc^2)
	\right)^{\frac{1}{2}} - \left( \lambda_{_{R}} (p-q) (\lambda_{_{R}}
	(p-q) + m c^2) \right)^{\frac{1}{2}} \right| }{ \lambda_{_{R}} (p-q)
	\lambda_{_{R}}( p ) } \\
	& \qquad + \frac{1}{\sqrt{2}} \frac{c }{ R} \frac{|q|}{mc^2}
	\\
	& \leq \frac{1}{\sqrt{2}} \frac{c |p| }{ R \lambda_{_{R}}( p ) \lambda_{_{R}} (p-q)
	} \left ( \frac{ | \lambda^2_{_{R}}( p ) - \lambda^2_{_{R}}
	(p-q) | } { \lambda_{_{R}} (p) +
	\lambda_{_{R}} (p-q)  } + mc^2
	\frac{| \lambda^2_{_{R}}( p ) - \lambda^2_{_{R}} (p-q) | } { (
	\lambda_{_{R}} (p) + \lambda_{_{R}} (p-q) )^2 } \right) \\
	& \qquad + \frac{1}{\sqrt{2}} \frac{|q|}{mc R}\\
	& \leq \frac{\sqrt{2}}{mcR} \frac{ | | p |^2 - |p-q|^2 | } {
	|p| + |p-q| } + \frac{1}{\sqrt{2}} \frac{|q|}{mc R} \leq
	 \frac{3 \sqrt{2}|q|}{2mc R}.
    \end{align*}    

    Therefore we may conclude that,
    \begin{equation*}
	| K_{R}(p, q)| \leq | K_{+,R}(p,q) \I_4| + \sum_{i=1}^{3} |
	K^{i}_{-,R}(p,q) \beta \alpha_i | \leq \frac{5\sqrt{2} |q|
	}{mc R}
    \end{equation*}
    and, in particular we get
    \begin{equation*}
	\sup_{(p,q) \in \R^3 \times \R^3} \frac{ | K_{R}(p, q)| }{1 +
	|q|} \leq \frac{C}{R} .
    \end{equation*}
    Moreover, since $ \hat{\chi}$ is the Fourier transform of the
    compact support function $\chi \in C_0^{\infty}(\R^3)$, it decays
    at infinity faster than any power, namely for any $\alpha >0$
    there exists a positive constant $C_{\alpha} >0$ such that
    \begin{equation}
	\label{eq:decay2}
	| \hat{\chi}(q)| \leq \frac{C_{\alpha} }{(1 + |q|)^{\alpha}}.
    \end{equation}  
 
    Then, by using H\"older inequality we get
    \begin{align*}
	& \norm{[ \chi_{_{R}} ,\fw^{-1}] \fw \psi}^2_{_{H^{1/2}}} =
	\int_{\R^3} dp \, (1 + R^{-1} |p|) \left| 
	\int_{\R^3} dq \, \hat{\chi}(q)K_{R}(p, q) \hat{
	\psi}_{_R}(p-q) \right|^2 \\
	& = \int_{\R^3} dp \, (1 + R^{-1}|p|) \left|\int_{\R^3} dq \,
	\hat{\chi}(q) K_{R}(p, q) \hat{ \psi}_{_R}(p-q) \right|^2\\
	& \leq \sup_{(p,q) \in \R^3 \times R^3} \frac{ | K_{R}(p,
	q)|^2 }{(1 + |q|)^2} \iint_{\R^3} dq_1 dq_2 (1 + |q_1|) |
	\hat{\chi}(q_1)| (1 + |q_2|) | \hat{\chi}(q_2) |\quad \times
	\\
	&  \int_{\R^3} dp \, (1 + \frac{|p-q_1|}{R} + \frac{
	|q_1|}{R})^{\frac{1}{2}} |\hat{ \psi}_{_R}(p-q_1 ) | (1 + \frac{|p-q_2|}{R} +
	\frac{|q_2|}{R})^{\frac{1}{2}} | \hat{ \psi}_{_R}(p-q_2 ) | \\
	& \leq \frac{C}{R^2} \left( \int_{\R^3} dq \, (1 + |q|) |
	\hat{\chi}(q)| \left( \int_{\R^3} dp \, (1 + \frac{|p-q|}{R} +
	\frac{|q|}{R}) | \hat{ \psi}_{_R}(p-q) |^2 \right)^{\frac{1}{2}}
	\right)^2
    \end{align*}    
    Note that,
    \begin{align*}
	\int_{\R^3} dp \, &(1 + R^{- 1}|p-q| + R^{- 1}|q|) | \hat{
	\psi}_{_R}(p-q) |^2 = \int_{\R^3} dp \, (1 + R^{- 1}|p-q| ) |
	\hat{ \psi}_{_R}(p-q) |^2 \\
	& \qquad \qquad + R^{- 1}|q| \int_{\R^3} dp \, | \hat{
	\psi}_{_R}(p-q) |^2 \\
	& \quad = \int_{\R^3} dp (1 + R^{- 1}|p| ) | \hat{
	\psi}_{_R}(p) |^2 + R^{- 1}|q| \int_{\R^3} dp \, | \hat{
	\psi}_{_R}(p) |^2 \\
	& \quad = \int_{\R^3} dp (1 + R^{- 1}|p| ) R^{- 3} | \hat{
	\psi}(R^{- 1} p) |^2 + R^{- 1}|q| \int_{\R^3} dp \, R^{- 3} |
	\hat{ \psi}(R^{- 1} p) |^2 \\
	& \quad = \norm{\psi}^2_{H^{1/2}} + R^{- 1}|q| | \psi|^2_{L^2}
	\leq (1 + R^{- 1}|q| ) \norm{\psi}^2_{H^{1/2}}
    \end{align*}    
    Hence, for $R >1$, we may conclude
    \begin{align*}
	\norm{[ \chi_{_{R}} ,\fw^{-1}]  & \fw \psi}_{_{H^{1/2}}} \\
	&\leq
	\frac{C}{R} \int_{\R^3} dq \, (1 + |q|) | \hat{\chi}(q)|
	\left( \int_{\R^3} dp \, (1 + \frac{|p-q|}{R} + \frac{|q|}{R}) |
	\hat{ \psi}_{_R}(p-q) |^2 \right)^{\frac{1}{2}} \\
	& \leq \frac{C}{R} \norm{\psi}_{H^{1/2}} \int_{\R^3} dq \, (1 +
	|q|) | \hat{\chi}(q)| (1 + R^{- 1}|q| )^{1/2} \\
	& \leq \frac{C}{R} \norm{\psi}_{H^{1/2}} \int_{\R^3} dq \, (1
	+ |q|)^{3/2} | \hat{\chi}(q)| \leq \frac{C}{R}
	\norm{\psi}_{H^{1/2}}.
    \end{align*}    
\end{proof}

\bibliographystyle{amsplain}

\end{document}